\documentclass[1 [leqno,11pt]{amsart}
\usepackage{amssymb, amsmath,latexsym,amsfonts,amsbsy, amsthm,mathtools,graphicx,CJKutf8,CJKnumb,CJKulem,color}
\usepackage{float}
\usepackage{hyperref}
\usepackage{enumerate}
\numberwithin{equation}{section}

\allowdisplaybreaks

\let\g=\gamma

\let\wt=\widetilde

\def\R{\mathbb R}

\def\Z{\mathbb Z}

\newcommand{\beq}{\begin{equation}}
\newcommand{\eeq}{\end{equation}}
\newcommand{\ben}{\begin{eqnarray}}
\newcommand{\een}{\end{eqnarray}}
\newcommand{\beno}{\begin{eqnarray*}}
\newcommand{\eeno}{\end{eqnarray*}}

\newtheorem{theorem}{Theorem}[section]

\newtheorem{lemma}[theorem]{Lemma}

\newtheorem{remark}[theorem]{Remark}
\newtheorem{Theorem}{Theorem}[section]

\newtheorem{Proposition}[Theorem]{Proposition}

\def\AAA{1500}\def\AAB{200}\def\AAC{270}\def\ABB{100}\def\ABC{230}\def\AAD{10^4}\def\ABD{8000}\def\ACD{3000}
\begin{document}
\title[Stability of the 1D screened Vlasov Poisson]
{Nonlinear stability of the one dimensional screened Vlasov
Poisson equation}
\author{Dongyi Wei}
\address{School of Mathematical Sciences, Peking University, Beijing 100871,  China}
\email{jnwdyi@pku.edu.cn}
\maketitle
\begin{abstract}
We study the asymptotic behavior of small data solutions to the screened Vlasov
Poisson(i.e. Vlasov-Yukawa) equation on $\R\times\R$ near vacuum. We show that for initial data small in Gevrey-2 regularity,  the 
derivative of the density of order $n$ decays like $(t+1)^{-n-1}$.
\end{abstract}
\section{Introduction}
In this paper, we study the following screened Vlasov-Poisson system (also called the Vlasov-Yukawa system, see \cite{Y}) for a particle density function $f:\R^d\times\R^d\times\R\to\R$:
\begin{align}
&\partial_t f+v\cdot\nabla_x f-q\nabla_x\phi\cdot\nabla_v f=0,\quad
(1-\Delta)\phi(x,t)=\rho(x,t)=\int_{\R^d}f(x,v,t)\mathrm{d}v.
\label{eq:VY}
\end{align}Here $q=\pm1$ corresponds to an attractive or repulsive force. Compared to the classical Vlasov-Poisson system, the standard Coulomb 
interaction kernel is replaced by screened interactions. For the case $d\geq2$, the asymptotic stability of \eqref{eq:VY} was proved in 
\cite{CHL,D,HRV}. The case $d=1$ is much more complicated, the asymptotic stability of \eqref{eq:VY} remains open, 
only the long-time stability of initially small and analytic solutions is known (see \cite{HRV}). 
In this paper, we focus on the case $d=1$ and solve this open problem. The main result is as follows.
\begin{theorem}\label{thm1}
Let $d=1$, if the initial data $ f_0(x,v)$ is smooth and satisfies \begin{align}\label{f03}f_0\in W^{4,1}(\R\times\R),\quad \| (\partial_x+\partial_v)^{n+1}f_0\|_{L_{x,v}^1}\leq (n!)^2/10^4,\quad \forall\ n\in\Z,\ n\geq0,\end{align} then the solution of \eqref{eq:VY} with initial data $ f(x,v,0)=f_0(x,v)$ is global in time and the density $\rho$ satisfies the following stability estimates:
\begin{align}\label{rho3}
|\partial_x^n\rho(x,t)|\le 3^n(n!)^2(t+1)^{-n-1}/10^3,\quad \forall\ t\geq0,\ x\in\R,\ n\in\Z,\ n\geq0.
\end{align}
\end{theorem}
Let us give some remarks.
\begin{itemize}

\item By time reversibility, if $\| (\partial_x-\partial_v)^{n+1}f_0\|_{L_{x,v}^1}\leq (n!)^2/10^4 $, then we can prove the decay estimate for $t\to-\infty$.

\item The initial condition can be stated as the smallness in a Gevrey-2 norm (see for example the definition in (A.14) in \cite{BM}).

\item Similar to \cite{IPWW}, we can prove the modified scattering.

\end{itemize}
Let $\tilde{f}_0(x,v)=f_0(x+v,v)$, then ${f}_0(x,v)=\tilde{f}_0(x-v,v)$, 
$\|\partial_v^{n+1}\tilde{f}_0\|_{L_{x,v}^1}=\|(\partial_x+\partial_v)^{n+1}f_0\|_{L_{x,v}^1}$.

For small initial data the dynamics of the Vlasov-Poisson system might be
 expected to be dominated by the free streaming part of the equation: $\partial_t f+v\cdot\nabla_x f=0 $.  In that case, integration along characteristics yields 
$f(x,v,t)=f_0(x-vt,v)=\tilde{f}_0(x-v(t+1),v) $, thus 
\begin{align*}
&\rho(x,t)=\int_{\R}f(x,v,t)\mathrm{d}v=\int_{\R}\tilde{f}_0(x-v(t+1),v)\mathrm{d}v=
\frac{1}{t+1}\int_{\R}\tilde{f}_0\left(x_0,\frac{x-x_0}{t+1}\right)\mathrm{d}x_0,\\
&\partial_x^n\rho(x,t)=\frac{1}{(t+1)^{n+1}}\int_{\R}(\partial_v^n\tilde{f}_0)\left(x_0,\frac{x-x_0}{t+1}\right)\mathrm{d}x_0.
\end{align*}Here we assume $d=1$. For $z\in\R$ we have 
\begin{align}\label{L1}
&|\partial_v^n\tilde{f}_0(x,z)|=\left|\int_{-\infty}^z\partial_v^{n+1}\tilde{f}_0(x,v)\mathrm{d}v\right|
\leq\int_{-\infty}^z|\partial_v^{n+1}\tilde{f}_0(x,v)|\mathrm{d}v\leq\int_{\R}|\partial_v^{n+1}\tilde{f}_0(x,v)|\mathrm{d}v.
\end{align}Then (using \eqref{L1} for $x=x_0$, $z=\frac{x-x_0}{t+1}$)\begin{align*}
|\partial_x^n\rho(x,t)|&\leq\frac{1}{(t+1)^{n+1}}\int_{\R}\left|(\partial_v^n\tilde{f}_0)\left(x_0,\frac{x-x_0}{t+1}\right)\right|\mathrm{d}x_0\\
&\leq\frac{1}{(t+1)^{n+1}}\int_{\R}\int_{\R}|\partial_v^{n+1}\tilde{f}_0(x_0,v)|\mathrm{d}v\mathrm{d}x_0\\
&=\frac{\|\partial_v^{n+1}\tilde{f}_0\|_{L_{x,v}^1}}{(t+1)^{n+1}}=\frac{\|(\partial_x+\partial_v)^{n+1}f_0\|_{L_{x,v}^1}}{(t+1)^{n+1}}.
\end{align*}Thus using $\tilde{f}_0$ instead of $f_0$, we don't need to discuss the case $t\leq 1$ and $t\geq1$ separately.

For \eqref{eq:VY}, the proof relies on the iteration frame work in \cite{BD} for Vlasov-Poisson equations, and in \cite{GS1,GS2,GS3,WY} for Vlasov-Maxwell equations. The solution to \eqref{eq:VY} can be constructed via the standard iteration scheme by starting with $ \rho^{(0)}=0$.
Given $ \rho^{(n-1)}$ define $ \phi^{(n-1)}$ as the solution to $(1-\Delta)\phi^{(n-1)}=\rho^{(n-1)}$, i.e. 
$\phi^{(n-1)}(t)=\frac{1}{2}\mathrm{e}^{-|x|}*\rho^{(n-1)}(t)$ and define $f^{(n)}$ as the solution of the linear Vlasov equation (here we restrict to the case $d=1$)\begin{align*}
&\partial_t f^{(n)}+v\partial_x f^{(n)}-q\partial_x\phi^{(n-1)}\partial_v f^{(n)}=0,\quad f^{(n)}(x,v,0)=f_0(x,v).
\end{align*}
Then we define the associated density $\rho^{(n)}(x,t)=\int_{\R}f^{(n)}(x,v,t)\mathrm{d}v$. Now a key step is to give a uniform estimate of  $\rho^{(n)}$. 
For notational ease, in the following, we take $\rho=\rho^{(n-1)} $, $\phi=\phi^{(n-1)} $, $f=f^{(n)} $, $\rho_*=\rho^{(n)} $. Then we have
\begin{align}
\label{f1}&\partial_t f+v\partial_x f-q\partial_x\phi\partial_v f=0,\quad (1-\Delta)\phi=\rho,\quad f(x,v,0)=f_0(x,v),\\
\label{rho*}& \rho_*(x,t)=\int_{\R}f(x,v,t)\mathrm{d}v.
\end{align}
Let $\varphi_n(t)=\mathrm{e}^{\frac{(n-2)\sqrt{t}}{n+\sqrt{t}}} $ for $n\geq2$ and $\varphi_0(t)=\varphi_1(t)=1 $. Let $ \gamma(t)=0.01\ln(t+1)+0.99t+1$ then we have 
$ \gamma(0)=\gamma'(0)=1$, $0.99<\gamma(t)/(t+1)\leq1$, $\gamma''(t)=-(t+1)^{-2}/100$.
\begin{Proposition}\label{prop1}
Assume that $ f_0(x,v)$ is smooth and satisfies \eqref{f03}, $ \rho$ satisfies
\begin{align}\label{rho2}
&|\partial_x^n\rho(x,t)|\leq \frac{\varphi_n(t)(n!)^2}{\AAA\gamma(t)^{n+1}},\quad \forall\ n\in\Z,\ n\geq0,\ t\geq0.
\end{align}Let $f,\phi$ solve \eqref{f1}, and $\rho_* $ be defined in \eqref{rho*}. Then we have\begin{align*}
&|\partial_x^n\rho_*(x,t)|\leq \frac{\varphi_n(t)(n!)^2}{\ACD\gamma(t)^{n+1}},\quad \forall\ n\in\Z,\ n\geq0,\ t\geq0.
\end{align*}
\end{Proposition}
\begin{proof}[Proof of Theorem \ref{thm1}]By Proposition \ref{prop1} and induction we have\begin{align}\label{rho4}
&|\partial_x^n\rho^{(k)}(x,t)|\leq \frac{\varphi_n(t)(n!)^2}{\ACD\gamma(t)^{n+1}},\quad \forall\ n\in\Z,\ n\geq0,\ k\in\Z_+, \ t\geq0.
\end{align}Using \eqref{rho4} and the same arguments as in \cite{BD,CHL}, we can prove that $(f^{(k)},\rho^{(k)},\phi^{(k)})$ is a Cauchy\\ sequence converging to a $C^1$ solution $(f,\rho,\phi)$ to \eqref{eq:VY}. Then by \eqref{rho4} and the Gagliardo–\\Nirenberg inequality, we have $\rho^{(k)}\to\rho$ in $L^{\infty}([0,T];C^n(\R))$
for all $T>0$, $n\in\Z_+$, thus\begin{align*}
&|\partial_x^n\rho(x,t)|\leq \frac{\varphi_n(t)(n!)^2}{\ACD\gamma(t)^{n+1}},\quad \forall\ n\in\Z,\ n\geq0, \ t\geq0.
\end{align*}And this implies Theorem \ref{thm1} by noting that $ \varphi_n(t)\leq \mathrm{e}^{n}\leq 2.97^{n}$, $ \gamma(t)\geq0.99(t+1)$ and
\begin{align*}
&\frac{\varphi_n(t)(n!)^2}{\ACD\gamma(t)^{n+1}}\leq \frac{2.97^{n}(n!)^2}{\ACD\cdot0.99^{n+1}(t+1)^{n+1}}=
\frac{3^{n}(n!)^2}{2970(t+1)^{n+1}}\leq\frac{3^{n}(n!)^2}{10^3(t+1)^{n+1}}.
\end{align*}\end{proof}
It remains to prove Proposition \ref{prop1}. Motivated by \cite{CHL,HRV}, we use the following auxiliary boundary value problem:\begin{align*}
\frac{\mathrm{d}}{\mathrm{d}s}X(s)=V(s),\quad \frac{\mathrm{d}}{\mathrm{d}s}V(s)=-q(\partial_x\phi)(X(s),s),\quad X(t)=x,\quad X(0)-V(0)=x_0.
\end{align*}More precisely, we should write $X(s)=X(s;x,x_0,t)$, $V(s)=V(s;x,x_0,t)$. Then we have 
$$f(x,V(t;x,x_0,t),t)=f_0(x_0+V(0;x,x_0,t),V(0;x,x_0,t))=\tilde{f}_0(x_0,V(0;x,x_0,t)),$$ and the density function $\rho_*$ defined in \eqref{rho*} can be
 represented as\begin{align*}
&\rho_*(x,t)=\int_{\R}f(x,v,t)\mathrm{d}v=\int_{\R}\tilde{f}_0(x_0,V(0;x,x_0,t))\frac{\partial V(t;x,x_0,t)}{\partial x_0}\mathrm{d}x_0.
\end{align*}Let $w(x,x_0,t)=V(t;x,x_0,t)$, $w_0(x,x_0,t)=V(0;x,x_0,t)$ then we have\begin{align}\label{rho1}
&\rho_*(x,t)=\int_{\R}\tilde{f}_0(x_0,w_0(x,x_0,t))\frac{\partial w}{\partial x_0}(x,x_0,t)\mathrm{d}x_0.
\end{align}Thanks to \eqref{rho1}, Proposition \ref{prop1} follows from the following (see \eqref{f02} in Lemma \ref{lem1.8}) \begin{align}\label{f02a}
&\int_{\R}|\partial_x^n[\tilde{f}_0(x_0,w_0(x,x_0,t))\partial_{x_0}w(x,x_0,t)]|\mathrm{d}x_0\leq\frac{(n!)^2\varphi_n(t)}{\ACD\gamma(t)^n}.
\end{align}The proof of 
\eqref{f02a} is based on  the estimate of  the characteristic in section 3, especially \eqref{X1}, \eqref{X2} in Lemma \ref{lem1.6} and \eqref{X3}, \eqref{X4} in Lemma \ref{lem1.8}.
\section{Estimate for derivatives of composed functions and weight functions}
The expression of derivatives of composed functions is
 given by the Fa\'a di Bruno formula:\begin{align}\label{Fg1}
&(F\circ g)^{(n)}=\sum_{*}\frac{n!}{m_1!\cdots m_n!}F^{(k)}\circ g\cdot\prod_{j=1}^n(g^{(j)}/j!)^{m_j}.
\end{align}Here $*$ denotes the sum over all $n$-tuples of non-negative integers $m_1,\cdots,m_n$ satisfying $\sum_{j=1}^njm_j=n$, and $k:=\sum_{j=1}^nm_j.$ 
The Fa\'a di Bruno formula is very useful when dealing with Gevrey regularity, see for example \cite{BM,IJ1,IJ2,IPWW2,MZ,MV}. 

It is very important to estimate the coefficients in \eqref{Fg1}. As the proof is very technical, if the reader is
 not interested in the details, you can skip the proof in this section.
\begin{lemma}\label{lem1.1}
Assume that $n\in\Z_+$, $m_j\in\Z$, $m_j\geq0$,  $\sum_{j=1}^njm_j=n$,  $\sum_{j=1}^nm_j=k$, $t\geq 0$, let $s=2k-m_1$, then we have (here $\prod_{j=2}^n=1 $, $\sum_{j=2}^n=0 $ for the case $n=1$)
\begin{align}\label{n1}
&\frac{(k!)^2}{m_1!n!}\prod_{j=2}^n(j!/2)^{m_j}\leq (s/n)^{(s-2)/2}\prod_{j=2}^n(j/n)^{(j-2)m_j/2},\\
\label{n2}&\frac{(k!)^2\varphi_k(t)}{m_1!n!\varphi_n(t)}\prod_{j=2}^n(j!\varphi_j(t)/2)^{m_j}\leq \prod_{j=2}^n(j/n)^{(j-2)m_j/4},\\
\label{n3}&\sum_{j=2}^n(j/n)^{(j-2)/4}\leq15,\\
\label{n4}&\sum_{*}\frac{(k!)^2\varphi_k(t)}{m_1!\cdots m_n!}\prod_{j=2}^n(j!\varphi_j(t)/200)^{m_j}\leq n!\varphi_n(t)\mathrm{e}^{0.15}.
\end{align} 
\end{lemma}\begin{remark}\label{rem1}
Recall that $\varphi_n(t)=\mathrm{e}^{\frac{(n-2)\sqrt{t}}{n+\sqrt{t}}} $ for $n\geq2$ and $\varphi_0(t)=\varphi_1(t)=1 $, then we have
$\varphi_{n+1}(t)/\varphi_n(t)=\mathrm{e}^{\frac{(\sqrt{t}+2)\sqrt{t}}{(n+\sqrt{t})(n+1+\sqrt{t})}}\geq1 $ for $n\geq2$ and $\varphi_{n+1}(t)/\varphi_n(t)=1 $ for $n=0,1$. Thus $\varphi_n(t) $ is increasing in $n$. As $\varphi_n(t)=\mathrm{e}^{n-2-\frac{(n-2)n}{n+\sqrt{t}}} $ for $n\geq2$, $\varphi_n(t) $ is increasing in $t$.
\end{remark} We need the following inequality
\begin{align}\label{log}
&\ln \frac{a}{b}\geq \frac{2(a-b)}{a+b},\quad \forall a\geq b>0.\end{align}
In fact let $x=\frac{a-b}{a+b}$, then $x\in[0,1)$, \eqref{log} becomes $\ln \frac{1+x}{1-x}\geq 2x $ and follows from 
$\ln \frac{1+x}{1-x}=\sum_{n=0}^{+\infty}\frac{2x^{2n+1}}{2n+1} $. We also need the following auxiliary lemma.
\begin{lemma}\label{lem1.1a}
If $n\geq j\geq2$, $n\geq3$, $n,j\in\Z$ then
\begin{align}\label{j1}
&j!/2\leq (n!/2)^{(j-2)/(n-2)}(j/n)^{(j-2)/2},\\
\label{j2}&\varphi_j(t)\leq \varphi_n(t)^{(j-2)/(n-2)}(j/n)^{-(j-2)/4}.\end{align}\end{lemma}
\begin{proof}
If $j=2$, as $\varphi_2(t)=1$, \eqref{j1}, \eqref{j2} become $1\leq 1$, $1\leq 1$, and are clearly true. Now we assume $n\geq j\geq3$.
Then  \eqref{j1} becomes $(j!/2)^{1/(j-2)}\leq (n!/2)^{1/(n-2)}(j/n)^{1/2} $. Let $a_j=(j!/2)^{1/(j-2)}$ then \eqref{j1} becomes
\begin{align}\label{j4}
&a_j/\sqrt{j}\leq a_n/\sqrt{n},\quad \forall\ n\geq j\geq3,\ n,j\in\Z.\end{align}
By induction, it is enough to prove that \begin{align}\label{k4}
&a_{n-1}/\sqrt{n-1}\leq a_n/\sqrt{n},\quad \forall\ n\in\Z,\ n\geq4.\end{align}
For $n\in\Z,\ n\geq4$ we have (using the Bernoulli inequality $1-kx\leq(1-x)^k$ for $x=1/n$)\begin{align*}
(a_{n-1}/a_n)^{(n-2)(n-3)}=&\frac{((n-1)!/2)^{n-2}}{(n!/2)^{n-3}}=\frac{(n-1)!/2}{n^{n-3}}=\prod_{k=1}^{n-3}\frac{n-k}{n}\leq \prod_{k=1}^{n-3}(1-1/n)^k\\=&(1-1/n)^{\sum_{k=1}^{n-3}k}=(1-1/n)^{(n-2)(n-3)/2},\\
a_{n-1}/a_n\leq&(1-1/n)^{1/2}=\sqrt{n-1}/\sqrt{n},
\end{align*}which implies \eqref{k4}, hence \eqref{j4} and \eqref{j1}. Recall that $\varphi_k(t)=\mathrm{e}^{\frac{(k-2)\sqrt{t}}{k+\sqrt{t}}} $ for $k\geq2$ we have\begin{align*}
&\ln \varphi_j(t)-\ln\big\{\varphi_n(t)^{(j-2)/(n-2)}\big\}=\frac{(j-2)\sqrt{t}}{j+\sqrt{t}}-\frac{(j-2)\sqrt{t}}{n+\sqrt{t}}=
\frac{(j-2)(n-j)\sqrt{t}}{(j+\sqrt{t})(n+\sqrt{t})}\\
=&\frac{(j-2)(n-j)\sqrt{t}}{(\sqrt{n}+\sqrt{j})^2\sqrt{t}+(\sqrt{t}-\sqrt{nj})^2}\leq \frac{(j-2)(n-j)}{(\sqrt{n}+\sqrt{j})^2}=\frac{(j-2)(\sqrt{n}-\sqrt{j})}{\sqrt{n}+\sqrt{j}}.
\end{align*}By \eqref{log} we have $\ln \frac{\sqrt{n}}{\sqrt{j}}\geq \frac{2(\sqrt{n}-\sqrt{j})}{\sqrt{n}+\sqrt{j}} $, thus
\begin{align*}
&\ln \varphi_j(t)-\ln\big\{\varphi_n(t)^{(j-2)/(n-2)}\big\}\leq\frac{(j-2)(\sqrt{n}-\sqrt{j})}{\sqrt{n}+\sqrt{j}}
\leq \frac{j-2}{2}\ln \frac{\sqrt{n}}{\sqrt{j}}=\ln\big\{(j/n)^{-(j-2)/4}\big\},
\end{align*}which implies \eqref{j2}. This completes the proof.
\end{proof}
\begin{proof}[Proof of Lemma \ref{lem1.1}]
As $m_j\geq0$,  $\sum_{j=1}^njm_j=n$,  $\sum_{j=1}^nm_j=k$, we have $nk=\sum_{j=1}^nnm_j\geq\sum_{j=1}^njm_j=n$, $k\geq1$, $k\in\Z_+$. And if $k=1$ then 
$m_j=0$ for $1\leq j<n$ and $m_n=1$.

If $n=1$ then $m_1=1$, $k=1$, $s=1$, $\varphi_1(t)=1$, \eqref{n1}--\eqref{n4} become $1\leq1$, $1\leq1$, $0\leq15$, $1\leq\mathrm{e}^{0.15}$, and are clearly true.

If $n=2$ then $(m_1,m_2,k,s)=(2,0,2,2)$ or $(0,1,1,2)$, $\varphi_1(t)=\varphi_2(t)=1$. \eqref{n1}--\eqref{n4} become 
$1\leq1$, $1/2\leq1$; $1\leq1$, $1/2\leq1$; $1\leq15$; $2+1/100\leq2\mathrm{e}^{0.15}$. They are clearly true.

In the following we assume $n\geq3$. As $k=\sum_{j=1}^nm_j\geq m_1$, $s=2k-m_1\geq k$ we have
\begin{align}\label{k1}
&\frac{(k!)^2}{m_1!n!}\prod_{j=2}^n(j!/2)^{m_j}=\frac{(k!)^2}{m_1!s!}\frac{s!/2}{n!/2}\prod_{j=2}^n(j!/2)^{m_j},\quad
\frac{(k!)^2}{m_1!s!}={k\choose k-m_1}\Big/{s\choose k-m_1}\leq1.\end{align}
Notice that $k\in\Z_+$, $s\geq k$; if $k>1$ then $s\geq k\geq2$; if $k=1$ then 
$m_j=0$ for $1\leq j<n$ and $m_n=1$, thus $m_1=0$ (as $n>1$), $s=2k-m_1=2$. Therefore we  always have $s\geq2$ (for $n>1$).
By \eqref{j1} in Lemma \ref{lem1.1a} we have 
\begin{align}\notag
\frac{s!/2}{n!/2}\prod_{j=2}^n(j!/2)^{m_j}\leq& \frac{(n!/2)^{(s-2)/(n-2)}(s/n)^{(s-2)/2}}{n!/2}\prod_{j=2}^n\big\{(n!/2)^{(j-2)/(n-2)}(j/n)^{(j-2)/2}\big\}^{m_j}\\
\label{s1}=&(n!/2)^{(s-2+\sum_{j=2}^n(j-2)m_j)/(n-2)-1}(s/n)^{(s-2)/2}\prod_{j=2}^n(j/n)^{(j-2)m_j/2}.\end{align}
As $\sum_{j=1}^njm_j=n$, $\sum_{j=1}^nm_j=k$, we have
\begin{align}\label{s2}
s-2+\sum_{j=2}^n(j-2)m_j=s-2+m_1+\sum_{j=1}^n(j-2)m_j=s-2+m_1+n-2k=n-2.\end{align}
Then \eqref{n1} follows from \eqref{k1}, \eqref{s1} and \eqref{s2}.
Notice that
\begin{align}\label{k2}
&\frac{(k!)^2\varphi_k(t)}{m_1!n!\varphi_n(t)}\prod_{j=2}^n(j!\varphi_j(t)/2)^{m_j}=
\bigg\{\frac{(k!)^2}{m_1!n!}\prod_{j=2}^n(j!/2)^{m_j}\bigg\}\bigg\{\frac{\varphi_k(t)}{\varphi_n(t)}\prod_{j=2}^n\varphi_j(t)^{m_j}\bigg\}.\end{align}
As $s\geq k$, we have $1\leq\varphi_k(t)\leq \varphi_s(t)$ (see Remark \ref{rem1}); as $s\geq2$ by \eqref{j2} in Lemma \ref{lem1.1a} and  \eqref{s2} we have
\begin{align}\label{k3}
&\frac{\varphi_k(t)}{\varphi_n(t)}\prod_{j=2}^n\varphi_j(t)^{m_j}\leq 
\frac{\varphi_s(t)}{\varphi_n(t)}\prod_{j=2}^n\varphi_j(t)^{m_j}\\ \notag
\leq& \frac{\varphi_n(t)^{(s-2)/(n-2)}(s/n)^{-(s-2)/4}}{\varphi_n(t)}\prod_{j=2}^n\big\{\varphi_n(t)^{(j-2)/(n-2)}(j/n)^{-(j-2)/4}\big\}^{m_j}\\ \notag
=&\varphi_n(t)^{(s-2+\sum_{j=2}^n(j-2)m_j)/(n-2)-1} (s/n)^{-(s-2)/4}\prod_{j=2}^n(j/n)^{-(j-2)m_j/4}\\ \notag
=& (s/n)^{-(s-2)/4}\prod_{j=2}^n(j/n)^{-(j-2)m_j/4}\leq (s/n)^{-(s-2)/2}\prod_{j=2}^n(j/n)^{-(j-2)m_j/4}.\end{align}
Then \eqref{n2} follows from \eqref{n1}, \eqref{k2} and \eqref{k3}.

For \eqref{n3}, we consider 2 cases. If $n\leq 16$ then\begin{align*}
&\sum_{j=2}^n(j/n)^{(j-2)/4}\leq\sum_{j=2}^n1=n-1\leq15.
\end{align*}If $n\geq 17$ we first claim that \begin{align}\label{j3}
&(j/n)^{(j-2)/4}\leq (6/7)^{j-2}+(7/8)^{n-j},\quad \forall\ 2\leq j\leq n,\ n\geq 17,\ n,j\in\Z.
\end{align}
In fact if $2\leq j\leq n/2$ then $(j/n)^{(j-2)/4}\leq (1/2)^{(j-2)/4}\leq (6/7)^{j-2} $ as $(1/2)^{1/4}\leq 6/7$. 

If $n/2< j\leq n$  then by \eqref{log}  we have\begin{align*}
&-\ln\big\{(j/n)^{(j-2)/4}\big\}=\frac{j-2}{4}\ln\frac{n}{j}\geq \frac{j-2}{4}\frac{2(n-j)}{n+j}=(n-j)\frac{j-2}{2(n+j)}.
\end{align*}
As $0<n/2< j$, $n,j\in\Z$ we have $0<n<2j$, $0<n\leq 2j-1$, as $n\geq 17$ we have $j>n/2\geq 17/2>8$, $j\geq9$, thus
\begin{align*}
&\frac{j-2}{2(n+j)}\geq \frac{j-2}{2(2j-1+j)}=\frac{1}{6+10/(j-2)}\geq \frac{1}{6+10/(9-2)}=\frac{7}{52}\geq \ln\frac{8}{7},\\
&-\ln\big\{(j/n)^{(j-2)/4}\big\}\geq (n-j)\frac{j-2}{2(n+j)}\geq (n-j)\ln\frac{8}{7}=-\ln\big\{(7/8)^{n-j}\big\},\\
&(j/n)^{(j-2)/4}\leq (7/8)^{n-j}.
\end{align*}
This completes the proof of \eqref{j3}. By \eqref{j3} we have (for $n\geq 17$)\begin{align*}
\sum_{j=2}^n(j/n)^{(j-2)/4}&\leq\sum_{j=2}^n\big\{(6/7)^{j-2}+(7/8)^{n-j}\big\}=7(1-(6/7)^{n-1})+8(1-(7/8)^{n-1})\\
&\leq7+8=15.
\end{align*}
Combining the cases $n\leq 16$ and $n\geq 17$ we conclude \eqref{n3}. By \eqref{n2} and \eqref{n3} we have
\begin{align*}
&\frac{1}{n!\varphi_n(t)}\sum_{*}\frac{(k!)^2\varphi_k(t)}{m_1!\cdots m_n!}\prod_{j=2}^n(j!\varphi_j(t)/200)^{m_j}
=\sum_{*}\frac{(k!)^2\varphi_k(t)}{m_1!n!\varphi_n(t)}\prod_{j=2}^n\frac{(j!\varphi_j(t)/2)^{m_j}}{100^{m_j}m_j!}\\
\leq&\sum_{*}\prod_{j=2}^n\frac{(j/n)^{(j-2)m_j/4}}{100^{m_j}m_j!}\leq \sum_{m_2=0}^{+\infty}\cdots\sum_{m_n=0}^{+\infty}\prod_{j=2}^n\frac{(j/n)^{(j-2)m_j/4}}{100^{m_j}m_j!}=
\prod_{j=2}^n\sum_{m_j=0}^{+\infty}\frac{(j/n)^{(j-2)m_j/4}}{100^{m_j}m_j!}\\
=&\prod_{j=2}^n\exp\left\{(j/n)^{(j-2)/4}/100\right\}=\exp\Bigg\{\sum_{j=2}^n(j/n)^{(j-2)/4}/100\Bigg\}\leq
\exp\left\{15/100\right\},
\end{align*}which implies \eqref{n4}. Here we used that for fixed $n$, under the restriction $\sum_{j=1}^njm_j=n$, $m_1$ is uniquely determined by $m_2,\cdots,m_n$.  This completes the proof.
\end{proof}
We also need the following estimates for the weight functions $\varphi_n(t) $. 
\begin{lemma}\label{lem1.2b}For $n\in\Z_+$, $t\geq0$  we have $\sum_{k=1}^{n}{n\choose k}^{-1}
\frac{\varphi_{n-k}(t)\varphi_k(t)}{\varphi_{n}(t)}\leq\frac{5}{3}.
$ For $n\in\Z$, $n\geq0$, $t\geq0$  we have $\sum_{k=0}^{n}{n\choose k}^{-1}
\frac{\varphi_{n-k}(t)\varphi_k(t)}{\varphi_{n}(t)}\leq\frac{8}{3}.
$\end{lemma}\begin{proof}
If $n=0$ then the second inequality becomes $1\leq \frac{8}{3}$ and is clearly true. If 
$n\in \Z_+$ then the second inequality follows from the first inequality by noting that 
${n\choose k}^{-1}\frac{\varphi_{n-k}(t)\varphi_k(t)}{\varphi_{n}(t)}|_{k=0}=\varphi_0(t)=1$. It remains to prove the first inequality. 
If $n\in\Z_+$, $n\leq 5$, then for $k\in\Z\cap[1,n]$ we have $k\leq 2$ or $n-k\leq 2$, by $\varphi_0(t)=\varphi_1(t)=\varphi_2(t)=1 $ and Remark \ref{rem1} we have $\frac{\varphi_{n-k}(t)\varphi_k(t)}{\varphi_{n}(t)}=\frac{\varphi_{n-k}(t)}{\varphi_{n}(t)}\leq1$ or $\frac{\varphi_{n-k}(t)\varphi_k(t)}{\varphi_{n}(t)}=\frac{\varphi_{k}(t)}{\varphi_{n}(t)}\leq1$, thus \begin{align*}
&\sum_{k=1}^{n}{n\choose k}^{-1}
\frac{\varphi_{n-k}(t)\varphi_k(t)}{\varphi_{n}(t)}\leq \sum_{k=1}^{n}{n\choose k}^{-1}=\left\{\begin{array}{ll}
                                                                                                 1\leq 5/3, & n=1, \\
                                                                                                 1/2+1=3/2\leq 5/3, & n=2, \\
                                                                                                 1/3+1/3+1= 5/3, & n=3, \\
                                                                                                 1/4+1/6+1/4+1= 5/3, & n=4, \\
                                                                                                 2(1/5+1/10)+1=1.6\leq 5/3, & n=5, 
                                                                                               \end{array}\right.
\end{align*}
which gives the first inequality for $n\leq 5$.  Now we assume $n\geq6$, we claim that
\begin{align}\label{n5}
&{n\choose k}^{-1}
\frac{\varphi_{n-k}(t)\varphi_k(t)}{\varphi_{n}(t)}\leq {n\choose 2}^{-1},\quad \forall\ n-2\geq k\geq2,\ n\geq6,\ n,k\in\Z.
\end{align}
For $n-2\geq k\geq2,\ n\geq6,\ n,k\in\Z $, let $j=n-k$ then $n-2\geq j\geq2$, $n-2\geq4$, by \eqref{j1} in Lemma \ref{lem1.1a} and 
$(j-2)+(k-2)=n-4 $ (as $j+k=n$) we have 
\begin{align}\label{n6}
&{n\choose 2}{n\choose k}^{-1}=\frac{n!}{2(n-2)!}\frac{j!k!}{n!}=\frac{(j!/2)(k!/2)}{(n-2)!/2}\\
\notag\leq&
\frac{((n-2)!/2)^{(j-2)/(n-4)}(j/(n-2))^{(j-2)/2}((n-2)!/2)^{(k-2)/(n-4)}(k/(n-2))^{(k-2)/2}}{(n-2)!/2}\\
\notag=&(j/(n-2))^{(j-2)/2}(k/(n-2))^{(k-2)/2}\leq (j/(n-2))^{(j-2)/4}(k/(n-2))^{(k-2)/4}.
\end{align}By \eqref{j2} in Lemma \ref{lem1.1a} and 
$(j-2)+(k-2)=n-4 $  we have \begin{align*}
\varphi_{j}(t)\varphi_k(t)&\leq \varphi_{n-2}(t)^{(j-2)/(n-4)}(j/(n-2))^{-(j-2)/4}\varphi_{n-2}(t)^{(k-2)/(n-4)}(k/(n-2))^{-(k-2)/4}\\
&=\varphi_{n-2}(t)(j/(n-2))^{-(j-2)/4}(k/(n-2))^{-(k-2)/4}.
\end{align*}Then by $\varphi_{n-2}(t)\leq \varphi_{n}(t) $ (see Remark \ref{rem1}) and $j=n-k$ we have\begin{align}\label{n7}
&\frac{\varphi_{n-k}(t)\varphi_k(t)}{\varphi_{n}(t)}\leq \frac{\varphi_{j}(t)\varphi_k(t)}{\varphi_{n-2}(t)}\leq (j/(n-2))^{-(j-2)/4}(k/(n-2))^{-(k-2)/4}.
\end{align}
Now \eqref{n5} follows from \eqref{n6} and \eqref{n7}. By $\varphi_0(t)=\varphi_1(t)=1 $ and Remark \ref{rem1} we have (i) if $k=1$ then $\frac{\varphi_{n-k}(t)\varphi_k(t)}{\varphi_{n}(t)}=\frac{\varphi_{n-k}(t)}{\varphi_{n}(t)}\leq1$; (ii) if $k=n-1$ or $k=n$ then $\frac{\varphi_{n-k}(t)\varphi_k(t)}{\varphi_{n}(t)}=\frac{\varphi_{k}(t)}{\varphi_{n}(t)}\leq1$. By (i)(ii) and \eqref{n5} we have
\begin{align*}
&\sum_{k=1}^{n}{n\choose k}^{-1}\frac{\varphi_{n-k}(t)\varphi_k(t)}{\varphi_{n}(t)}\leq {n\choose 1}^{-1}+
\sum_{k=2}^{n-2}{n\choose 2}^{-1}+{n\choose n-1}^{-1}+{n\choose n}^{-1}\\=&\frac{1}{n}+(n-3){n\choose 2}^{-1}+\frac{1}{n}+1=\frac{2}{n}+\frac{2(n-3)}{n(n-1)}+1=\frac{4(n-2)}{n(n-1)}+1\leq\frac{4}{n}+1\leq\frac{4}{6}+1=\frac{5}{3}.
\end{align*}This completes the proof of the first inequality, hence Lemma \ref{lem1.2b}.
\end{proof}
\begin{lemma}\label{lem1.2a}For $n\in\Z_+$, $t\geq0$  we have\begin{align}\label{t1}
&\int_0^t \min\left\{2\varphi_{n-1}(s),\frac{\varphi_{n+1}(s)(n(n+1))^2}{\gamma(s)^{2}}\right\}\frac{t-s}{\gamma(t)}\mathrm{d}s\leq \min\left\{\frac{50}{9}n^2\varphi_{n}(t),500\varphi_{n-1}(t)\right\}.\end{align}
\end{lemma}\begin{proof}
If $n=1$, as $\varphi_0(t)=\varphi_1(t)=\varphi_2(t)=1 $ and $\gamma(t)\geq0.99(t+1) $ we have\begin{align*}
&\int_0^t \min\left\{2\varphi_{n-1}(s),\frac{\varphi_{n+1}(s)(n(n+1))^2}{\gamma(s)^{2}}\right\}\frac{t-s}{\gamma(t)}\mathrm{d}s=\int_0^t \min\left\{2,\frac{2^2}{\gamma(s)^{2}}\right\}\frac{t-s}{\gamma(t)}\mathrm{d}s\\ \leq&\int_0^t \min\left\{2,\frac{2^2}{0.99^2(s+1)^{2}}\right\}\frac{t-s}{0.99(t+1)}\mathrm{d}s\leq \int_0^t \min\left\{2,\frac{2^2}{0.99^2(s+1)^{2}}\right\}\frac{1}{0.99}\mathrm{d}s\\
\leq& \int_0^{+\infty} \min\left\{2,\frac{2^2}{0.99^2(s+1)^{2}}\right\}\frac{1}{0.99}\mathrm{d}s\leq \int_0^{1}\frac{2}{0.99}\mathrm{d}s+\int_1^{+\infty}\frac{2^2}{0.99^3(s+1)^{2}}\mathrm{d}s\\
=&\frac{2}{0.99}+\frac{2}{0.99^3}<5.\end{align*}As $n=1$ and $\varphi_0(t)=\varphi_1(t)=1 $ we also have\begin{align*}
& \min\left\{\frac{50}{9}n^2\varphi_{n}(t),500\varphi_{n-1}(t)\right\}=\min\left\{\frac{50}{9},500\right\}=\frac{50}{9}>5.\end{align*} 
Then \eqref{t1} holds for $n=1$. If $n=2$, as $\varphi_1(t)=\varphi_2(t)=1 $,  and $\gamma(t)\geq0.99(t+1) $ we have\begin{align*}
&\int_0^t \min\left\{2\varphi_{n-1}(s),\frac{\varphi_{n+1}(s)(n(n+1))^2}{\gamma(s)^{2}}\right\}\frac{t-s}{\gamma(t)}\mathrm{d}s=\int_0^t \min\left\{2,\frac{6^2\varphi_{3}(s)}{\gamma(s)^{2}}\right\}\frac{t-s}{\gamma(t)}\mathrm{d}s\\
\leq&\int_0^t \min\left\{2,\frac{6^2\varphi_{3}(s)}{0.99^2(s+1)^{2}}\right\}\frac{t-s}{0.99(t+1)}\mathrm{d}s\leq \int_0^t \min\left\{2,\frac{6^2\varphi_{3}(s)}{0.99^2(s+1)^{2}}\right\}\frac{1}{0.99}\mathrm{d}s\\
\leq& \int_0^{+\infty} \min\left\{2,\frac{6^2\varphi_{3}(s)}{0.99^2(s+1)^{2}}\right\}\frac{1}{0.99}\mathrm{d}s\leq \int_0^{4}\frac{2}{0.99}\mathrm{d}s+\int_4^{+\infty}\frac{6^2\varphi_{3}(s)}{0.99^3(s+1)^{2}}\mathrm{d}s.\end{align*}
Let $F_1(t):=\varphi_3(t)/(t+1)$, as $\varphi_3(t)=\mathrm{e}^{\frac{\sqrt{t}}{3+\sqrt{t}}}\geq1$ we have $\varphi_3'(t)=\frac{3}{2\sqrt{t}(3+\sqrt{t})^2}\varphi_3(t)$ and \begin{align*}
&F_1'(t)=-\frac{\varphi_3(t)}{(t+1)^2}+\frac{3\varphi_3(t)}{2\sqrt{t}(3+\sqrt{t})^2(t+1)}
=-\frac{17\varphi_3(t)}{20(t+1)^2}-\frac{3(\sqrt{t}-2)(t-2\sqrt{t}+5)\varphi_3(t)}{20\sqrt{t}(3+\sqrt{t})^2(t+1)^2}.\end{align*}
Thus $F_1'(t)\leq-\frac{17\varphi_3(t)}{20(t+1)^2}  $ for $t\geq4$ and (using $ \mathrm{e}^{2/5}\leq3/2$)\begin{align*}
&\int_4^{+\infty}\frac{\varphi_{3}(s)}{(s+1)^{2}}\mathrm{d}s\leq \int_4^{+\infty}-\frac{20}{17}F_1'(s)\mathrm{d}s=\frac{20}{17}F_1(4)=\frac{20}{17}\frac{\varphi_3(4)}{4+1}=\frac{4}{17}\varphi_3(4)
=\frac{4}{17}\mathrm{e}^{2/5}\leq \frac{6}{17},\\
&\int_0^t \min\left\{2\varphi_{n-1}(s),\frac{\varphi_{n+1}(s)(n(n+1))^2}{\gamma(s)^{2}}\right\}\frac{t-s}{\gamma(t)}\mathrm{d}s\leq \int_0^{4}\frac{2}{0.99}\mathrm{d}s+\int_4^{+\infty}\frac{6^2\varphi_{3}(s)}{0.99^3(s+1)^{2}}\mathrm{d}s\\
&\leq \frac{8}{0.99}+\frac{6^2}{0.99^3}\frac{6}{17}\leq 8.1+13.1=21.2<22<200/9.\end{align*}
As $n=2$ and $\varphi_1(t)=\varphi_2(t)=1 $ we also have\begin{align*}
& \min\left\{\frac{50}{9}n^2\varphi_{n}(t),500\varphi_{n-1}(t)\right\}=\min\left\{\frac{200}{9},500\right\}=\frac{200}{9}.\end{align*} 
Then \eqref{t1} holds for $n=2$. Now we assume $n\geq3$. As $\varphi_{k}(s)\leq \varphi_{k}(t) $ for $0\leq s\leq t$ and $k\geq2$ (see Remark \ref{rem1}) we have (also using $\gamma(t)\geq0.99(t+1) $)\begin{align*}
&\int_0^t \min\left\{2\varphi_{n-1}(s),\frac{\varphi_{n+1}(s)(n(n+1))^2}{\gamma(s)^{2}}\right\}\frac{t-s}{\gamma(t)}\mathrm{d}s\\
\leq&\int_0^t \min\left\{2\varphi_{n-1}(t),\frac{\varphi_{n+1}(t)(n(n+1))^2}{0.99^2(s+1)^{2}}\right\}\frac{t-s}{0.99(t+1)}\mathrm{d}s\\
\leq&\int_0^t \min\left\{2\varphi_{n-1}(t),\frac{\varphi_{n+1}(t)(n(n+1))^2}{0.99^2(s+1)^{2}}\right\}\frac{1}{0.99}\mathrm{d}s\\
\leq&\int_0^{+\infty} \min\left\{2\varphi_{n-1}(t),\frac{\varphi_{n+1}(t)(n(n+1))^2}{0.99^2s^{2}}\right\}\frac{1}{0.99}\mathrm{d}s\\
=&\frac{2}{0.99}\sqrt{2\varphi_{n-1}(t)\frac{\varphi_{n+1}(t)(n(n+1))^2}{0.99^2}}=
\frac{2\sqrt{2}}{0.99^2}n(n+1)\sqrt{\varphi_{n-1}(t)\varphi_{n+1}(t)}.\end{align*}
Here we used that $\int_0^{+\infty} \min\{A,B/s^2\}\mathrm{d}s=2\sqrt{AB} $ for $A>0$, $B>0$. As $n\geq3$ and  $\varphi_{n+1}(t)/\varphi_n(t)=\mathrm{e}^{\frac{(\sqrt{t}+2)\sqrt{t}}{(n+\sqrt{t})(n+1+\sqrt{t})}}$,
$\varphi_{n}(t)/\varphi_{n-1}(t)=\mathrm{e}^{\frac{(\sqrt{t}+2)\sqrt{t}}{(n-1+\sqrt{t})(n+\sqrt{t})}}\leq\mathrm{e}$ (see Remark \ref{rem1}), we have $\varphi_{n+1}(t)/\varphi_n(t)\leq\varphi_{n}(t)/\varphi_{n-1}(t)$, $\varphi_{n-1}(t)\varphi_{n+1}(t)\leq \varphi_{n}(t)^2 $ and 
(using $ 2\sqrt{2}/0.99^2<3$)
\begin{align}\label{t2}
&\int_0^t \min\left\{2\varphi_{n-1}(s),\frac{\varphi_{n+1}(s)(n(n+1))^2}{\gamma(s)^{2}}\right\}\frac{t-s}{\gamma(t)}\mathrm{d}s\leq
\frac{2\sqrt{2}}{0.99^2}n(n+1)\sqrt{\varphi_{n-1}(t)\varphi_{n+1}(t)}\\
\notag\leq& 3n(n+1)\varphi_{n}(t)=3(1+1/n)n^2\varphi_{n}(t)\leq 3(1+1/3)n^2\varphi_{n}(t)=4n^2\varphi_{n}(t)\leq \frac{50}{9}n^2\varphi_{n}(t).\end{align}
It remains to prove that (for $n\in\Z$, $n\geq 3$, $t\geq0$)\begin{align}\label{t4}
&\int_0^t \min\left\{2\varphi_{n-1}(s),\frac{\varphi_{n+1}(s)(n(n+1))^2}{\gamma(s)^{2}}\right\}\frac{t-s}{\gamma(t)}\mathrm{d}s\leq
500\varphi_{n-1}(t).\end{align}
If $n\in\Z\cap[3,7]$, as $\varphi_{n}(t)/\varphi_{n-1}(t)\leq\mathrm{e}\leq 2.75$, by \eqref{t2} we have\begin{align*}
&\int_0^t \min\left\{2\varphi_{n-1}(s),\frac{\varphi_{n+1}(s)(n(n+1))^2}{\gamma(s)^{2}}\right\}\frac{t-s}{\gamma(t)}\mathrm{d}s\leq
 3n(n+1)\varphi_{n}(t)\\ \leq& 3\cdot7\cdot8\varphi_{n}(t) \leq 3\cdot7\cdot8\cdot2.75\varphi_{n-1}(t)=462\varphi_{n-1}(t)\leq 500\varphi_{n-1}(t).\end{align*}Then \eqref{t4} holds for $n\in\Z\cap[3,7]$. Now we assume $n\geq8$. 
 Let $g_n(t)=t\varphi_{n-1}(t)$, $h_n(t)=\varphi_{n-1}'(t)/\varphi_{n-1}(t)$, as 
 $\varphi_{n-1}(t)=\mathrm{e}^{\frac{(n-3)\sqrt{t}}{n-1+\sqrt{t}}} $ we have 
 $h_{n}(t)=\frac{(n-3)(n-1)}{2\sqrt{t}(n-1+\sqrt{t})^2} $ and \begin{align*}
&\varphi_{n-1}'(t)=h_{n}(t)\varphi_{n-1}(t),\quad \varphi_{n-1}''(t)=(h_{n}'(t)+h_n(t)^2)\varphi_{n-1}(t),\\
&g_n''(t)=t\varphi_{n-1}''(t)+2\varphi_{n-1}'(t)=(th_{n}'(t)+th_n(t)^2+2h_n(t))\varphi_{n-1}(t),\\
&th_{n}'(t)+2h_n(t)=-\frac{(n-3)(n-1)}{4\sqrt{t}(n-1+\sqrt{t})^2} -\frac{(n-3)(n-1)}{2(n-1+\sqrt{t})^3}
+\frac{(n-3)(n-1)}{\sqrt{t}(n-1+\sqrt{t})^2}\\=&\frac{(n-3)(n-1)(3(n-1)+\sqrt{t})}{4\sqrt{t}(n-1+\sqrt{t})^3}\geq0\Rightarrow 
 \frac{g_n''(t)}{\varphi_{n-1}(t)}\geq th_n(t)^2=\frac{(n-3)^2(n-1)^2}{4(n-1+\sqrt{t})^4}.
\end{align*}
If $0<t\leq (n+1)^2$ then we have (for $n\geq8$)\begin{align}\notag
& \frac{g_n''(t)}{\varphi_{n-1}(t)}\geq \frac{(n-3)^2(n-1)^2}{4(n-1+\sqrt{t})^4}\geq \frac{(n-3)^2(n-1)^2}{4(n-1+n+1)^4}=\frac{(n-3)^2(n-1)^2}{64n^4}\\
\notag=&(1-3/n)^2(1-1/n)^2/{8^2}\geq {(1-3/8)^2(1-1/8)^2}/{8^2}=(5/8)^2(7/8)^2/8^2\\
\notag=&5^2\cdot 7^2/8^6=(5\cdot 7/8^3)^2=(35/512)^2<(1/15)^2=1/225,\\
\label{vphi1}&2\varphi_{n-1}(t)\leq 450g_n''(t),\quad \forall\ n\in\Z,\ n\geq8,\ 0<t\leq (n+1)^2.
\end{align}Let $J_n(t)=t^2\cdot th_n(t)^2\varphi_{n-1}(t)/\varphi_{n+1}(t)>0$, as $ th_n(t)^2=\frac{(n-3)^2(n-1)^2}{4(n-1+\sqrt{t})^4}$, \\ $\partial_t(th_n(t)^2)=-\frac{2}{\sqrt{t}(n-1+\sqrt{t})}th_n(t)^2 $, $h_n(t)=\varphi_{n-1}'(t)/\varphi_{n-1}(t)$ and $h_{n+2}(t)=\varphi_{n+1}'(t)/\varphi_{n+1}(t)$, then 
 we have\begin{align*}
&\frac{J_n'(t)}{J_n(t)}=\frac{2}{t}-\frac{2}{\sqrt{t}(n-1+\sqrt{t})}+h_n(t)-h_{n+2}(t)\\
=&\frac{2(n-1)}{{t}(n-1+\sqrt{t})}+\frac{(n-3)(n-1)}{2\sqrt{t}(n-1+\sqrt{t})^2}-\frac{(n+1)(n-1)}{2\sqrt{t}(n+1+\sqrt{t})^2}\\
\geq&\frac{2(n-1)}{\sqrt{t}(n-1+\sqrt{t})^2}+\frac{(n-3)(n-1)}{2\sqrt{t}(n-1+\sqrt{t})^2}-\frac{(n+1)(n-1)}{2\sqrt{t}(n-1+\sqrt{t})^2}=0.
\end{align*}Thus $J_n(t)$ is increasing for $t>0$; if $t\geq (n+1)^2$ then $J_n(t)\geq J_n((n+1)^2)$. As \\
$ J_n(t)=t^2\cdot th_n(t)^2\varphi_{n-1}(t)/\varphi_{n+1}(t)=t^2\frac{(n-3)^2(n-1)^2}{4(n-1+\sqrt{t})^4}\varphi_{n-1}(t)/\varphi_{n+1}(t)$, and
\begin{align*}
&t^2\frac{(n-3)^2(n-1)^2}{4(n-1+\sqrt{t})^4}\bigg|_{t=(n+1)^2}=(n+1)^4\frac{(n-3)^2(n-1)^2}{4(n-1+n+1)^4}
=\frac{(n+1)^4(n-3)^2(n-1)^2}{64n^4},\\
&\varphi_{n-1}((n+1)^2)=\mathrm{e}^{\frac{(n-3)(n+1)}{n-1+n+1}}=\mathrm{e}^{\frac{n^2-2n-3}{2n}},\quad
\varphi_{n+1}((n+1)^2)=\mathrm{e}^{\frac{(n-1)(n+1)}{n+1+n+1}}=\mathrm{e}^{\frac{n-1}{2}},
\end{align*}we have (still for $n\geq8$; using $ 8^4/(5\cdot 63)=13+1/315<\sqrt{170}$ and $\mathrm{e}^{11/16}<2 $)
\begin{align*}
&J_n((n+1)^2)=\frac{(n+1)^4(n-3)^2(n-1)^2}{64n^4}\mathrm{e}^{\frac{n^2-2n-3}{2n}}/\mathrm{e}^{\frac{n-1}{2}}
=\frac{(n+1)^4(n-3)^2(n-1)^2}{64n^4\mathrm{e}^{\frac{n+3}{2n}}},\\
&\frac{J_n((n+1)^2)}{(n(n+1))^2}=\frac{(n+1)^2(n-3)^2(n-1)^2}{64n^6\mathrm{e}^{\frac{n+3}{2n}}}=
\frac{(1-3/n)^2(1-1/n^2)^2}{64\mathrm{e}^{(1+3/n)/2}}\\
\geq &\frac{(1-3/8)^2(1-1/8^2)^2}{8^2\mathrm{e}^{(1+3/8)/2}}=\frac{(5/8)^2(63/8^2)^2}{8^2\mathrm{e}^{11/16}}
=\frac{5^2\cdot 63^2}{8^8\mathrm{e}^{11/16}}=\frac{(5\cdot 63/8^4)^2}{\mathrm{e}^{11/16}}\geq 
\frac{1/170}{2}=\frac{1}{340}.
\end{align*}Thus if $t\geq (n+1)^2$, $n\geq8$ then $J_n(t)\geq J_n((n+1)^2)\geq (n(n+1))^2/340$ and 
\begin{align*}
&g_n''(t)\geq th_n(t)^2\varphi_{n-1}(t)=J_n(t)\varphi_{n+1}(t)/t^2\geq (n(n+1))^2\varphi_{n+1}(t)/(340t^2),\\
&\frac{\varphi_{n+1}(t)(n(n+1))^2}{\gamma(t)^{2}}\leq \frac{\varphi_{n+1}(t)(n(n+1))^2}{0.99^2(t+1)^{2}}\leq 
\frac{340t^2g_n''(t)}{0.99^2(t+1)^{2}}\leq 
\frac{340g_n''(t)}{0.99^2}\leq 
350g_n''(t),
\end{align*}(here we used $\gamma(t)\geq0.99(t+1) $), which along with \eqref{vphi1} implies 
\begin{align*}\min\left\{2\varphi_{n-1}(t),\frac{\varphi_{n+1}(t)(n(n+1))^2}{\gamma(t)^{2}}\right\}\leq 
450g_n''(t),\quad \forall\ n\in\Z,\ n\geq8,\ t>0.\end{align*}
Thus (note that $g_n\in C^1([0,+\infty))$, $g_n'(t)=\big(1+\frac{\sqrt{t}(n-3)(n-1)}{2(n-1+\sqrt{t})^2}\big)\varphi_{n-1}(t) $, $g_n'(0)=\varphi_{n-1}(0)=1$)
\begin{align*}
&\int_0^t \min\left\{2\varphi_{n-1}(s),\frac{\varphi_{n+1}(s)(n(n+1))^2}{\gamma(s)^{2}}\right\}\frac{t-s}{\gamma(t)}\mathrm{d}s\leq
\int_0^t  450g_n''(s)\frac{t-s}{\gamma(t)}\mathrm{d}s\\ =&
\frac{450}{\gamma(t)}\left(\int_0^t  g_n'(s)\mathrm{d}s-tg_n'(0)\right)=\frac{450}{\gamma(t)}
\left(g_n(t)-t\right)\leq \frac{450}{\gamma(t)}g_n(t)= \frac{450}{\gamma(t)}t\varphi_{n-1}(t)\\ \leq& 
\frac{450}{0.99(t+1)}t\varphi_{n-1}(t)\leq\frac{450}{0.99}\varphi_{n-1}(t)\leq500\varphi_{n-1}(t).\end{align*}
Then \eqref{t4} holds for $n\in\Z$, $n\geq8$. Now  \eqref{t1} (for $n\geq3$) follows from \eqref{t2} and \eqref{t4}.
\end{proof}
\section{Estimate of the density function}
We first estimate the potential $\phi$ in terms of the density $\rho$. As $d=1$, we have $\partial_x^n\rho=\partial_x^n(1-\Delta)\phi=(1-\Delta)\partial_x^n\phi=\partial_x^n\phi-\partial_x^{n+2}\phi $. By maximum principle we have \begin{align*}
&\|\partial_x^n\phi(t)\|_{L^{\infty}}\leq \|\partial_x^n\rho(t)\|_{L^{\infty}},\quad 
\|\partial_x^{n+2}\phi(t)\|_{L^{\infty}}\leq \|\partial_x^n\phi(t)\|_{L^{\infty}}+\|\partial_x^n\rho(t)\|_{L^{\infty}}\leq 2\|\partial_x^n\rho(t)\|_{L^{\infty}}.
\end{align*}Thus\begin{align}\label{phi1}
&\|\partial_x^{n+2}\phi(t)\|_{L^{\infty}}\leq \min(2\|\partial_x^n\rho(t)\|_{L^{\infty}},\|\partial_x^{n+2}\rho(t)\|_{L^{\infty}}),\quad
\forall\ n\in\Z,\ n\geq0.
\end{align}\begin{lemma}\label{lem1.2}Assume \eqref{rho2}, then $\|\partial_x^{2}\phi(t)\|_{L^{\infty}}\leq-\gamma''(t)/\gamma(t) $, and for $n\in\Z_+$  we have\begin{align*}
&\int_0^t \|\partial_x^{n+1}\phi(s)\|_{L^{\infty}}\gamma(s)^n\frac{t-s}{\gamma(t)}\mathrm{d}s\leq \min\left\{\frac{\varphi_n(t)(n!)^2}{\AAC},\frac{\varphi_{n-1}(t)((n-1)!)^2}{3}\right\}.
\end{align*}\end{lemma}
\begin{proof}
By \eqref{rho2} and \eqref{phi1} (with $n$ replaced by $n-1$) we have
\begin{align}\label{phi3}
&\|\partial_x^{n+1}\phi(t)\|_{L^{\infty}}\leq \min\left\{\frac{2\varphi_{n-1}(t)((n-1)!)^2}{\AAA\gamma(t)^{n}},\frac{\varphi_{n+1}(t)((n+1)!)^2}{\AAA\gamma(t)^{n+2}}\right\}.
\end{align}
In particular, taking $n=1$ in \eqref{phi3}, using $0.99<\gamma(t)/(t+1)\leq1$, $\gamma''(t)=-(t+1)^{-2}/100$ and $\varphi_{2}(t)=1$ we have\begin{align*}
&\|\partial_x^{2}\phi(t)\|_{L^{\infty}}\leq \frac{\varphi_{2}(t)(2!)^2}{\AAA\gamma(t)^{3}}\leq
\frac{4}{\AAA\cdot0.99^2(t+1)^{2}\gamma(t)}\leq\frac{1}{100(t+1)^{2}\gamma(t)}
=-\frac{\gamma''(t)}{\gamma(t)}.
\end{align*}By \eqref{phi3} and Lemma \ref{lem1.2a} we have \begin{align*}
&\int_0^t \|\partial_x^{n+1}\phi(s)\|_{L^{\infty}}\gamma(s)^n\frac{t-s}{\gamma(t)}\mathrm{d}s\\
\leq& \int_0^t \min\left\{\frac{2\varphi_{n-1}(s)((n-1)!)^2}{\AAA},\frac{\varphi_{n+1}(s)((n+1)!)^2}{\AAA\gamma(s)^{2}}\right\}\frac{t-s}{\gamma(t)}\mathrm{d}s\\
=&\frac{((n-1)!)^2}{\AAA}\int_0^t \min\left\{2\varphi_{n-1}(s),\frac{\varphi_{n+1}(s)(n(n+1))^2}{\gamma(s)^{2}}\right\}\frac{t-s}{\gamma(t)}\mathrm{d}s\\ 
\leq&\frac{((n-1)!)^2}{\AAA}\min\left\{\frac{50}{9}n^2\varphi_{n}(t),500\varphi_{n-1}(t)\right\}\\
=&\min\left\{\frac{\varphi_n(t)(n!)^2}{\AAC},\frac{\varphi_{n-1}(t)((n-1)!)^2}{3}\right\}.
\end{align*}This completes the proof.
\end{proof}
Now we estimate the characteristic curves $X$. By definition we have $\partial_s^2X=-q(\partial_1\phi)(X,s), $ taking derivatives $\partial_x $ and $\partial_{x_0} $ then
\begin{align}\label{X5}
\partial_s^2\partial_xX=-q(\partial_1^2\phi)(X,s)\partial_xX,\quad \partial_s^2\partial_{x_0}X=-q(\partial_1^2\phi)(X,s)\partial_{x_0}X.
\end{align}
Here to avoid the ambiguity of $\partial_x $, we write $\partial_1\phi $ instead of $\partial_x\phi $. 
Then by Lemma \ref{lem1.2}, $h(s):=-q(\partial_1^2\phi)(X(s;x,x_0,t),s)$ satisfies $|h(s)|\leq -\gamma''(s)/\gamma(s)$ (assume \eqref{rho2}). 
Let $y_1(s)=\partial_xX(s) $, $y_2(s)=\partial_{x_0}X(s) $, then we have $y_i''(s)=h(s)y_i(s)$ ($i=1,2$) with boundary conditions 
$y_1(0)=y_1'(0)$, $ y_1(t)=1$, $y_2(0)=y_2'(0)+1$, $y_2(t)=0$. So we need to estimate the ODE of the type $y''(s)=h(s)y(s)$.
\begin{lemma}\label{lem1.3}
Assume that $t>0$ and $|h(s)|\leq -\gamma''(s)/\gamma(s)$ for $s\geq0$.\\ (i) If $y_1(s)$ solves $y_1''(s)=h(s)y_1(s)$ for $s\in[0,t]$, $y_1(0)=y_1'(0)>0$ then\\ $ 0<y_1(s)\leq\gamma(s)y_1(t)/\gamma(t) $ for $s\in[0,t]$.\\ (ii) If $y_2(s)$ solves $y_2''(s)=h(s)y_2(s)$ for $s\in[0,t]$, $y_2(0)=y_2'(0)+1$, $y_2(t)=0$ then\\ $ 0\leq y_2(s)\leq(t-s)/\gamma(t) $ for $s\in[0,t]$.\\
(iii) If $y(s)$ solves $y''(s)=h(s)y(s)+F(s)$ for $s\in[0,t]$, $y(0)=y'(0)$, $y(t)=0$ then
\begin{align*}
\sup_{s\in[0,t]}\frac{|y(s)|}{\gamma(s)}\leq \int_{0}^{t}|F(s)|\frac{t-s}{\gamma(t)}\mathrm{d}s,\quad
\sup_{s\in[0,t)}\frac{|y(s)|}{t-s}\leq 
\int_{0}^{t}|F(s)|\frac{\gamma(s)}{\gamma(t)}\mathrm{d}s.
\end{align*}
\end{lemma}

\begin{proof}
	\begin{enumerate}[(i)]
		\item Let $s^*:=\sup\{s_0\in(0,t): y_1(s)>0\text{ for all }s\in[0, s_0]\}$, 
then by $y_1(0)>0$ and the continuity of $y_1$ we have $s^*\in(0, t]$ and $y_1(s)>0$ for $s\in[0,s_*)$. 
Let $\phi_1(s):=y_1(s)/\g(s)$ for $s\in[0,t]$, then 
		\begin{align}\label{A1}\phi_1'(s)=\frac{A_1(s)}{\g(s)^2},\quad A_1(s):=y_1'(s)\g(s)-\g'(s)y_1(s),\quad\forall\ s\in[0,t].	\end{align}
		By the equation of $y_1$ and the hypothesis on $h$, we have
		\begin{align*}
			A_1'(s)=y_1''(s)\g(s)-\g''(s)y_1(s)=y_1(s)\left(-\g''(s)+h(s)\g(s)\right)\geq0\quad\forall\ s\in[0, s^*).
		\end{align*}
		Hence $A_1(s)\geq A_1(0)=0$ for $s\in[0, s^*]$, recalling $\g(0)=\g'(0)=1$ and the assumption $y_1(0)=y_1'(0)$. As a consequence, $\phi_1'(s)\geq0$ on $s\in[0, s^*]$, so $\phi_1$ is increasing on $[0, s^*]$, $\phi_1(s^*)\geq\phi_1(0)=y_1(0)>0$. If $s^*\in(0, t)$, then $y_1(s^*)=0$, $\phi_1(s^*):=y_1(s^*)/\g(s^*)=0$ which  is a contradiction.
 Therefore, we have $s^*=t$, hence $y_1(s)>0$ for $s\in[0, t)$, $\phi_1$ is increasing on $[0, t]$ and 
 $ 0<y_1(s)=\g(s)\phi_1(s)\leq\gamma(s)\phi_1(t)=\gamma(s)y_1(t)/\gamma(t) $ for $s\in[0,t]$.
		
		\item We first show that $y_2(s)>0$ for all $s\in[0, t)$. Indeed, we compute
		\begin{align}\label{A2}
\left(\frac{y_2}{y_1}\right)'(s)=\frac{A_2(s)}{y_1(s)^2},\quad A_2(s):=y_2'(s)y_1(s)-y_1'(s)y_2(s),\quad \forall\ s\in[0,t].\end{align}
		It follows from the equations of $y_1$ and $y_2$ that $A_2'(s)=y_2''(s)y_1(s)-y_1''(s)y_2(s)=h(s)y_2(s)y_1(s)-h(s)y_1(s)y_2(s)=0$ for all $s\in[0,t]$, thus 
		\begin{align}\label{A2a}
			&A_2(s)=A_2(0)=y_2'(0)y_1(0)-y_1'(0)y_2(0)\\ \notag=&\left(y_2(0)-1\right)y_1(0)-y_1(0)y_2(0)=-y_1(0)<0
		\end{align}
		for all $s\in[0,t]$. Hence (as $y_2(t)=0$)
		\begin{equation*}
			y_2(s)=-y_1(s)\int_s^t\frac{A_2(\tau)}{y_1(\tau)^2}\,\mathrm d\tau=y_1(s)\int_s^t\frac{y_1(0)}{y_1(\tau)^2}\,\mathrm d\tau,\quad \forall\ s\in[0,t].
		\end{equation*}
		Therefore $y_2(s)>0$ for all $s\in[0, t)$. To prove $y_2(s)\leq(t-s)/\g(t)$, we introduce
		\begin{equation}\label{tg}
			\wt\g(s):=\g(s)\int_s^t\frac{\mathrm d\tau}{\g(\tau)^2}>0, \quad \forall\ s\in[0,t],
		\end{equation}
		and we are going to show that
		\begin{equation}\label{Eq.Claim}
			y_2(s)\leq \wt\g(s)\leq\frac{t-s}{\g(t)},\quad\forall\ s\in[0,t].
		\end{equation}
		Firstly, we let $\phi_2(s):=y_2(s)/\wt\g(s)$ for $s\in[0,t)$, then 
		\begin{align}\label{A3}\phi_2'(s)=\frac{A_3(s)}{\wt\g(s)^2},\quad A_3(s):=y_2'(s)\wt\g(s)-y_2(s)\wt\g'(s),\quad\forall\ s\in[0,t).\end{align}
		By the definition of $\wt\g$ we have $\wt\g''=\g''\wt\g/\g$, hence
		\[A_3'(s)=y_2''(s)\wt\g(s)-y_2(s)\wt\g''(s)=y_2(s)\wt\g(s)\left(h(s)-\frac{\g''(s)}{\g(s)}\right)\geq 0,\quad\forall\ s\in[0, t),\]
		where we also used $y_2''=hy_2$, $|h(s)|\leq -\frac{\g''(s)}{\g(s)}$ and the fact that $y_2(s)>0$ for $s\in[0, t)$. Hence $A_3(s)\leq A_3(t)=0$ for $s\in[0, t)$ (as $y_2(t)=0=\wt\g(t)$), and thus $\phi_2'(s)\leq 0$ for $s\in[0, t)$. So we have 
$\phi_2(s)\leq \phi_2(0)$, $\forall\ s\in[0, t)$ and
		\[y_2(s)\leq \frac{y_2(0)}{\wt\g(0)}\wt\g(s),\quad\forall\ s\in[0, t].\]
		Now we claim that $y_2(0)\leq \wt\g(0)$. Indeed, by the definition of $\wt\g$ and the fact $\g(0)=\g'(0)=1$ we compute that $\wt\g'(0)=\wt\g(0)-1$, then
		\begin{align*}
			A_3(0)=y_2'(0)\wt\g(0)-y_2(0)\wt\g'(0)=(y_2(0)-1)\wt\g(0)-y_2(0)\left(\wt\g(0)-1\right)=y_2(0)-\wt\g(0),
		\end{align*}
		hence $A_3(0)\leq 0$ implies $y_2(0)\leq \wt\g(0)$. This proves the first inequality in \eqref{Eq.Claim}. To show the last inequality in \eqref{Eq.Claim}, we let
		\begin{equation*}
			F(s):=\wt\g(s)+\frac s{\g(t)}=\g(s)\int_s^t\frac{\mathrm d\tau}{\g(\tau)^2}+\frac s{\g(t)},\quad\forall\ s\in[0, t].
		\end{equation*}
		It is a direct computation that
		\begin{align*}
			F'(s)=\g'(s)\int_s^t\frac{\mathrm d\tau}{\g(\tau)^2}-\frac1{\g(s)}+\frac1{\g(t)},\quad F''(s)=\g''(s)\int_s^t\frac{\mathrm d\tau}{\g(\tau)^2},\quad\forall\ s\in[0,t].
		\end{align*}
		Since $\g''<0, \g>0$, we have $F''(s)<0$ for all $s\in[0,t]$, hence $F'(s)\geq F'(t)=0$ for all $s\in[0,t]$, and therefore $F(s)\leq F(t)=t/\g(t)$ for $s\in[0,t]$, as desired.
		
		\item Let $y_1$ solve the initial value problem 
$y_1''(s)=h(s)y_1(s)$ for $s\in[0,t]$ and $y_1(0)=y_1'(0)=1$; let $y_2$ be given by (ii). Then
		the solution $y$ is given by (note that $A_2(s):=y_2'(s)y_1(s)-y_1'(s)y_2(s)=A_2(0)=-y_1(0)=-1$, see \eqref{A2}, \eqref{A2a})
		\begin{equation}\label{Eq.SolODE}
			y(s)=-y_2(s)\int_0^sF(\tau)y_1(\tau)\,\mathrm d\tau-y_1(s)\int_s^tF(\tau)y_2(\tau)\,\mathrm d\tau=-\int_0^tK(s,\tau)F(\tau)\,\mathrm d\tau,
		\end{equation}
		where $K(s,\tau):=y_1(\tau)y_2(s)\mathbf 1_{\tau\in[0,s]}+y_1(s)y_2(\tau)\mathbf 1_{\tau\in[s,t]}$ for all $s,\tau\in [0,t]$.
It suffices to show that
		\begin{equation}\label{Eq.K_bound}
			\left|K(s, \tau)\right|\leq \min\left\{\frac{\g(s)}{\g(t)}(t-\tau),\frac{\g(\tau)}{\g(t)}(t-s)\right\}\quad\forall\ s,\tau\in[0, t].
		\end{equation}
        Indeed, since $A_1(s)\geq 0$, $A_3(s)\leq 0$, $y_1(s)>0$, $y_2(s)>0$, $\g(s)>0$ and $\wt\g(s)>0$ for $s\in[0,t)$, we have 
        (see \eqref{A1}, \eqref{A3} for the definitions of $A_1(s)$, $A_3(s)$)
        \begin{equation*}
            \frac{y_1'(s)}{y_1(s)}\geq\frac{\g'(s)}{\g(s)},\quad \frac{y_2'(s)}{y_2(s)}\leq \frac{\wt\g'(s)}{\wt\g(s)},\quad\forall\ s\in[0,t).
        \end{equation*}
        On the other hand, by $A_2(s)=-y_1(0)=-1$ (see \eqref{A2}, \eqref{A2a}), $\g(0)=\g'(0)=1$ and the definition of $\wt\g$ in \eqref{tg} we have
        \begin{equation*}
            y_1'(s)y_2(s)-y_1(s)y_2'(s)=1,\quad \g'(s)\wt\g(s)-\g(s)\wt\g'(s)=1,\quad\forall\ s\in[0,t].
        \end{equation*}
        Hence
        \begin{align*}
            \frac1{y_1(s)y_2(s)}&=\frac{y_1'(s)y_2(s)-y_1(s)y_2'(s)}{y_1(s)y_2(s)}=\frac{y_1'(s)}{y_1(s)}-\frac{y_2'(s)}{y_2(s)}\\
            &\geq \frac{\g'(s)}{\g(s)}-\frac{\wt\g'(s)}{\wt\g(s)}=\frac{\g'(s)\wt\g(s)-\g(s)\wt\g'(s)}{\g(s)\wt\g(s)}=\frac{1}{\g(s)\wt\g(s)}
        \end{align*}
        for all $s\in[0, t)$, then using $y_2(t)=0=\wt\g(t)$ we have 
        \begin{equation}\label{Eq.y_1y_2_compare}
            y_1(s)y_2(s)\leq \g(s)\wt\g(s),\quad\forall \ s\in[0,t].
        \end{equation}
        Now, we are ready to prove \eqref{Eq.K_bound}. Using the symmetry property $K(s,\tau)=K(\tau,s) $ by the definition of $K$, 
        we can assume that 
        $0\leq \tau\leq s\leq t$. 
        then by the increasing of $\phi_1=y_1/\g$, \eqref{Eq.y_1y_2_compare} and \eqref{Eq.Claim}, we have
        \[0\leq K(s,\tau)=y_1(\tau)y_2(s)\leq \g(\tau)\frac{y_1(s)}{\g(s)}y_2(s)\leq\g(\tau)\wt\g(s)\leq \g(\tau)\frac{t-s}{\g(t)}\leq \frac{\g(s)}{\g(t)}(t-\tau),\]
        where in the last inequality we have also used the increasing of $\g$. 
        This proves \eqref{Eq.K_bound}. Therefore, using \eqref{Eq.SolODE} and \eqref{Eq.K_bound} we obtain
        \[\frac{|y(s)|}{\g(s)}\leq \int_0^t\frac{|K(s,\tau)|}{\g(s)}|F(\tau)|\,\mathrm d\tau\leq \frac1{\g(t)}\int_0^t|F(\tau)|(t-\tau)\,\mathrm d\tau,\quad\forall\ s\in[0,t];\]
        \[\frac{|y(s)|}{t-s}\leq \int_0^t\frac{|K(s,\tau)|}{t-s}|F(\tau)|\,\mathrm d\tau\leq \frac1{\g(t)}\int_0^t|F(\tau)|\g(\tau)\,\mathrm d\tau,\quad\forall\ s\in[0,t).\]
	\end{enumerate}This completes the proof.
\end{proof}

\begin{lemma}\label{lem1.4}Assume \eqref{rho2}, then for $0\leq s\leq t$ we have\begin{align*}
&0<\partial_xX(s;x,x_0,t)\leq \gamma(s)/\gamma(t),\quad 0\leq\partial_{x_0}X(s;x,x_0,t)\leq (t-s)/\gamma(t).
\end{align*}\end{lemma}\begin{proof}
For fixed $ x,x_0,t$, let $h(s):=-q(\partial_1^2\phi)(X(s;x,x_0,t),s)$, $y_1(s):=\partial_xX(s;x,x_0,t) $, then by Lemma \ref{lem1.2}, we have $|h(s)|\leq -\gamma''(s)/\gamma(s)$ for $s\geq0$, by \eqref{X5} we have $y_1''(s)=h(s)y_1(s)$ for $s\in[0,t]$ and  $y_1(0)=y_1'(0)$, $ y_1(t)=1$. Here we used \begin{align}\label{X6}
X(t;x,x_0,t)=x,\quad X(0;x,x_0,t)-V(0;x,x_0,t)=x_0,\quad V=\partial_{s}X.
\end{align}
If $y_1(0)=y_1'(0)=0$ then by the uniqueness of the ODE $y''(s)=h(s)y(s)$ we have $y(s)\equiv0$ for $s\in[0,t]$, which contradicts $ y_1(t)=1$.

If $y_1(0)=y_1'(0)<0$ then we can apply Lemma \ref{lem1.3} (i) with $y_1$ replaced by $-y_1$ to deduce that $-y_1(s)>0$ for $s\in[0,t]$, which contradicts $ y_1(t)=1$.

So we must have $y_1(0)=y_1'(0)>0$ then by Lemma \ref{lem1.3} (i) we have $ 0<y_1(s)\leq\gamma(s)y_1(t)/\gamma(t)=\gamma(s)/\gamma(t) $ for $s\in[0,t]$, which implies the first inequality.

Let $y_2(s):=\partial_{x_0}X(s;x,x_0,t) $, then by \eqref{X5} and \eqref{X6} we have $y_2''(s)=h(s)y_2(s)$ for $s\in[0,t]$ and  $y_2(0)=y_2'(0)+1$, $ y_2(t)=0$. By Lemma \ref{lem1.3} (ii) we have $ 0\leq y_2(s)\leq(t-s)/\gamma(t) $ for $s\in[0,t]$, which implies the second inequality.
\end{proof}\begin{lemma}\label{lem1.5}Assume \eqref{rho2}. For fixed $n\in\Z$, $n\geq 2$ if 
\begin{align}\label{X7}
|\partial_x^jX(s;x,x_0,t)|\leq \frac{\varphi_j(t)(j!)^2\gamma(s)}{\AAB\gamma(t)^{j}},\quad \forall\ j\in\Z,\ 2\leq j<n,\ 0\leq s\leq t,
\end{align}
then we have\begin{align*}
&\int_0^t|\partial_x^n[(\partial_1\phi)(X,s)]-(\partial_1^2\phi)(X,s)\partial_x^nX|\frac{t-s}{\gamma(t)}\mathrm{d}s
\leq\frac{(n!)^2\varphi_n(t)}{\ABC\gamma(t)^n},\quad \forall\ t\geq0.
\end{align*}\end{lemma}\begin{proof}By the Fa\'a di Bruno formula we have\begin{align*}
&|\partial_x^n[(\partial_1\phi)(X,s)]-(\partial_1^2\phi)(X,s)\partial_x^nX|
\leq\sum_{**}\frac{n!}{m_1!\cdots m_n!}|(\partial_1^{k+1}\phi)(X,s)|\cdot\prod_{j=1}^n|\partial_x^jX/j!|^{m_j}.
\end{align*}Here $**$ denotes the sum over all $n$-tuples of non-negative integers $m_1,\cdots,m_n$ satisfying $\sum_{j=1}^njm_j=n$, $m_n=0$ 
and $k:=\sum_{j=1}^nm_j.$ Then by Lemma \ref{lem1.4}, \eqref{X7}, Lemma \ref{lem1.2} and \eqref{n4} in Lemma \ref{lem1.1} we have 
\begin{align*}
&\int_0^t|\partial_x^n[(\partial_1\phi)(X,s)]-(\partial_1^2\phi)(X,s)\partial_x^nX|\frac{t-s}{\gamma(t)}\mathrm{d}s
\\ \leq&\sum_{**}\frac{n!}{m_1!\cdots m_n!}\int_0^t|(\partial_1^{k+1}\phi)(X,s)|\bigg(\prod_{j=1}^n|\partial_x^jX/j!|^{m_j}\bigg)\frac{t-s}{\gamma(t)}\mathrm{d}s\\
\leq&\sum_{**}\frac{n!}{m_1!\cdots m_n!}\int_0^t\|\partial_1^{k+1}\phi(s)\|_{L^{\infty}}\left|\frac{\gamma(s)}{\gamma(t)}\right|^{m_1}\bigg(
\prod_{j=2}^n\left|\frac{\varphi_j(t)j!\gamma(s)}{\AAB\gamma(t)^{j}}\right|^{m_j}\bigg)\frac{t-s}{\gamma(t)}\mathrm{d}s\\
=&\sum_{**}\frac{n!}{m_1!\cdots m_n!}\int_0^t\|\partial_1^{k+1}\phi(s)\|_{L^{\infty}}\gamma(s)^k\frac{t-s}{\gamma(t)}\mathrm{d}s
\cdot\frac{1}{\gamma(t)^n}\cdot
\prod_{j=2}^n\left|\frac{\varphi_j(t)j!}{\AAB}\right|^{m_j}\\
\leq&\sum_{**}\frac{n!}{m_1!\cdots m_n!}\frac{\varphi_k(t)(k!)^2}{\AAC}\cdot\frac{1}{\gamma(t)^n}\cdot\prod_{j=2}^n\left|\frac{\varphi_j(t)j!}{\AAB}\right|^{m_j}
\leq\frac{(n!)^2\varphi_n(t)}{\AAC\gamma(t)^n}\mathrm{e}^{0.15}\leq\frac{(n!)^2\varphi_n(t)}{\ABC\gamma(t)^n}.
\end{align*}Here we used $\sum_{**}\leq \sum_{*} $ and \eqref{n4} in the last step. This completes the proof.
\end{proof}\begin{lemma}\label{lem1.6}Assume \eqref{rho2},
then we have\begin{align}\label{X1}
&|\partial_x^nX(s;x,x_0,t)|\leq \frac{\varphi_n(t)(n!)^2\gamma(s)}{\AAB\gamma(t)^{n}},\quad \forall\ n\in\Z,\ n\geq2,\ 0\leq s\leq t,\\ \label{X2}
&|\partial_xw_0(x,x_0,t)|\leq \frac{1}{\gamma(t)},\quad |\partial_x^nw_0(x,x_0,t)|\leq \frac{\varphi_n(t)(n!)^2}{\AAB\gamma(t)^{n}},\quad \forall\ n\in\Z,\ n\geq2,\ t\geq0.
\end{align}Moreover for $n\in\Z$, $n\geq0$, $t\geq0$ we have\begin{align}\label{f01}
&\int_{\R}|\partial_x^n[\tilde{f}_0(x_0,w_0(x,x_0,t))]|\mathrm{d}x_0\leq\frac{(n!)^2\varphi_n(t)}{\ABD\gamma(t)^n},\\ \label{phi2}
&\int_0^t|\partial_x^n[(\partial_1^2\phi)(X,s)]|\gamma(s)\frac{t-s}{\gamma(t)}\mathrm{d}s
\leq\frac{(n!)^2\varphi_n(t)}{2\gamma(t)^n}.
\end{align}\end{lemma}\begin{proof}
We first prove \eqref{X1} by induction on $n$, it is enough to prove \eqref{X1} assuming \eqref{X7}.
\footnote{When $n=2$, \eqref{X7} is nothing, i.e. automatically satisfied.}
As $\partial_s^2X=-q(\partial_1\phi)(X,s), $  we have $\partial_s^2\partial_x^nX=-q\partial_x^n[(\partial_1\phi)(X,s)]$.
For fixed $ x,x_0,t$ and $n\geq2$, let $h(s):=-q(\partial_1^2\phi)(X(s;x,x_0,t),s)$, $y(s):=\partial_x^nX(s;x,x_0,t) $, then we have
$y''(s)=h(s)y(s)+F(s)$ for $s\in[0,t]$ with $F(s):=-q(\partial_x^n[(\partial_1\phi)(X,s)]-(\partial_1^2\phi)(X,s)\partial_x^nX)$.

 By Lemma \ref{lem1.2}, we have $|h(s)|\leq -\gamma''(s)/\gamma(s)$ for $s\geq0$, by \eqref{X6} we have $y(0)=y'(0)$, $y(t)=0$.
 By Lemma \ref{lem1.3} (iii) and Lemma \ref{lem1.5} (assuming \eqref{X7}) we have \begin{align*}
\sup_{s\in[0,t]}\frac{|y(s)|}{\gamma(s)}&\leq \int_{0}^{t}|F(s)|\frac{t-s}{\gamma(t)}\mathrm{d}s=
\int_0^t|\partial_x^n[(\partial_1\phi)(X,s)]-(\partial_1^2\phi)(X,s)\partial_x^nX|\frac{t-s}{\gamma(t)}\mathrm{d}s\\
&\leq\frac{(n!)^2\varphi_n(t)}{\ABC\gamma(t)^n}\leq \frac{\varphi_n(t)(n!)^2}{\AAB\gamma(t)^{n}}.
\end{align*}This implies \eqref{X1} (assuming \eqref{X7}). By induction, we complete the proof of \eqref{X1}.

Recall that $w_0(x,x_0,t)=V(0;x,x_0,t)$, by \eqref{X6} we have $\partial_x^nw_0(x,x_0,t)=\partial_x^nX(0;x,x_0,t)$ for $n\in\Z_+$, then 
\eqref{X2} follows from Lemma \ref{lem1.4}, \eqref{X1} (and $\gamma(0)=1 $) by taking $s=0$.

For $n\in\Z_+$, by the Fa\'a di Bruno formula and \eqref{X2} we have
\begin{align}\notag
|\partial_x^n[\tilde{f}_0(x_0,w_0(x,x_0,t))]|&\leq\sum_{*}\frac{n!}{m_1!\cdots m_n!}|(\partial_v^{k}\tilde{f}_0)(x_0,w_0)|\cdot\prod_{j=1}^n|\partial_x^jw_0/j!|^{m_j}\\
\label{f04}&\leq\sum_{*}\frac{n!}{m_1!\cdots m_n!}|(\partial_v^{k}\tilde{f}_0)(x_0,w_0)|\cdot\frac{1}{\gamma(t)^{m_1}}\cdot
\prod_{j=2}^n\left|\frac{\varphi_j(t)j!}{\AAB\gamma(t)^{j}}\right|^{m_j}\\
\notag&=\sum_{*}\frac{n!}{m_1!\cdots m_n!}|(\partial_v^{k}\tilde{f}_0)(x_0,w_0)|\cdot\frac{1}{\gamma(t)^{n}}\cdot
\prod_{j=2}^n\left|\frac{\varphi_j(t)j!}{\AAB}\right|^{m_j}.
\end{align}By \eqref{L1} we have (for all $k\in\Z$, $k\geq0$)\begin{align}
\label{f05}&\int_{\R}|(\partial_v^{k}\tilde{f}_0)(x_0,w_0(x,x_0,t))|\mathrm{d}x_0
\leq\int_{\R}\int_{\R}|\partial_v^{k+1}\tilde{f}_0(x_0,v)|\mathrm{d}v\mathrm{d}x_0\\
\notag=&\|\partial_v^{k+1}\tilde{f}_0\|_{L_{x,v}^1}=\|(\partial_x+\partial_v)^{k+1}f_0\|_{L_{x,v}^1}\leq (k!)^2/\AAD.
\end{align}By \eqref{f04}, \eqref{f05}, \eqref{n4} in Lemma \ref{lem1.1} and $\varphi_k(t)\geq1 $ we have\begin{align*}
\int_{\R}|\partial_x^n[\tilde{f}_0(x_0,w_0(x,x_0,t))]|\mathrm{d}x_0&\leq\sum_{*}\frac{n!(k!)^2}{\AAD m_1!\cdots m_n!}\cdot\frac{1}{\gamma(t)^{n}}\cdot
\prod_{j=2}^n\left|\frac{\varphi_j(t)j!}{\AAB}\right|^{m_j}\\
&\leq \frac{(n!)^2\varphi_n(t)}{\AAD \gamma(t)^n}\mathrm{e}^{0.15}\leq\frac{(n!)^2\varphi_n(t)}{\ABD\gamma(t)^n},
\end{align*}which is \eqref{f01} for $n\in\Z_+$. For $n=0$, by \eqref{f05} for $k=0$ and $\varphi_0(t)=1 $ we have\begin{align*}
\int_{\R}|\partial_x^n[\tilde{f}_0(x_0,w_0(x,x_0,t))]|\mathrm{d}x_0&\leq \frac{1}{\AAD}\leq \frac{1}{\ABD}=\frac{(n!)^2\varphi_n(t)}{\ABD\gamma(t)^n},
\end{align*}\eqref{f01} is still true. This completes the proof of \eqref{f01}.

For $n\in\Z_+$, by the Fa\'a di Bruno formula and \eqref{X1} we have
\begin{align*}
|\partial_x^n[(\partial_1^2\phi)(X,s)]|&\leq\sum_{*}\frac{n!}{m_1!\cdots m_n!}|(\partial_1^{k+2}\phi)(X,s)|\cdot\prod_{j=1}^n|\partial_x^jX/j!|^{m_j}\\
&\leq\sum_{*}\frac{n!}{m_1!\cdots m_n!}\|\partial_1^{k+2}\phi(s)\|_{L^{\infty}}\cdot\left|\frac{\gamma(s)}{\gamma(t)}\right|^{m_1}\cdot
\prod_{j=2}^n\left|\frac{\varphi_j(t)j!\gamma(s)}{\AAB\gamma(t)^{j}}\right|^{m_j}\\ &=\sum_{*}\frac{n!}{m_1!\cdots m_n!}\|\partial_1^{k+2}\phi(s)\|_{L^{\infty}}\frac{\gamma(s)^k}{\gamma(t)^n}\cdot
\prod_{j=2}^n\left|\frac{\varphi_j(t)j!}{\AAB}\right|^{m_j}.
\end{align*}
Then by Lemma \ref{lem1.2} and \eqref{n4} in Lemma \ref{lem1.1} we have\begin{align*}
&\int_0^t|\partial_x^n[(\partial_1^2\phi)(X,s)]|\gamma(s)\frac{t-s}{\gamma(t)}\mathrm{d}s\\
\leq&\sum_{*}\frac{n!}{m_1!\cdots m_n!}\int_0^t\|\partial_1^{k+2}\phi(s)\|_{L^{\infty}}\gamma(s)^{k+1}\frac{t-s}{\gamma(t)}\mathrm{d}s\cdot\frac{1}{\gamma(t)^n}\cdot
\prod_{j=2}^n\left|\frac{\varphi_j(t)j!}{\AAB}\right|^{m_j}\\
\leq&\sum_{*}\frac{n!}{m_1!\cdots m_n!}\frac{\varphi_k(t)(k!)^2}{3}\cdot\frac{1}{\gamma(t)^n}\cdot\prod_{j=2}^n\left|\frac{\varphi_j(t)j!}{\AAB}\right|^{m_j}
\leq\frac{(n!)^2\varphi_n(t)}{3\gamma(t)^n}\mathrm{e}^{0.15}\leq\frac{(n!)^2\varphi_n(t)}{2\gamma(t)^n},
\end{align*}which is \eqref{phi2} for $n\in\Z_+$. For $n=0$, by Lemma \ref{lem1.2} we have\begin{align*}
\int_0^t|\partial_x^n[(\partial_1^2\phi)(X,s)]|\gamma(s)\frac{t-s}{\gamma(t)}\mathrm{d}s\leq \int_0^t\|\partial_1^{2}\phi(s)\|_{L^{\infty}}\gamma(s)\frac{t-s}{\gamma(t)}\mathrm{d}s\leq\frac{\varphi_0(t)}{3}
\leq\frac{(n!)^2\varphi_n(t)}{2\gamma(t)^n},
\end{align*}\eqref{phi2} is still true. This completes the proof of \eqref{phi2}.
\end{proof}\begin{lemma}\label{lem1.7}Assume \eqref{rho2}. For fixed $n\in\Z_+$ if \begin{align}\label{X8}
|\partial_x^j\partial_{x_0}X(s;x,x_0,t)|\leq \frac{\varphi_j(t)(j!)^2(t-s)}{\gamma(t)^{j+1}},\quad \forall\ j\in\Z,\ 0\leq j<n,\ 0\leq s\leq t,
\end{align} 
then we have\begin{align*}
&\int_0^t|\partial_x^n[(\partial_1^2\phi)(X,s)\partial_{x_0}X]-(\partial_1^2\phi)(X,s)\partial_x^n\partial_{x_0}X|\frac{\gamma(s)}{\gamma(t)}\mathrm{d}s
\leq\frac{(n!)^2\varphi_n(t)}{\gamma(t)^{n+1}},\quad \forall\ t\geq0.
\end{align*}\end{lemma}\begin{proof}
By the Leibniz formula, \eqref{X8}, \eqref{phi2} in Lemma \ref{lem1.6} and  Lemma \ref{lem1.2b} we have
\begin{align*}
&\int_0^t|\partial_x^n[(\partial_1^2\phi)(X,s)\partial_{x_0}X]-(\partial_1^2\phi)(X,s)\partial_x^n\partial_{x_0}X|
\frac{\gamma(s)}{\gamma(t)}\mathrm{d}s\\
\leq&\sum_{k=1}^{n}{n\choose k}
\int_0^t|\partial_x^k[(\partial_1^2\phi)(X,s)]||\partial_x^{n-k}\partial_{x_0}X|\frac{\gamma(s)}{\gamma(t)}\mathrm{d}s\\
\leq&\sum_{k=1}^{n}{n\choose k}
\int_0^t|\partial_x^k[(\partial_1^2\phi)(X,s)]|\frac{\varphi_{n-k}(t)((n-k)!)^2(t-s)}{\gamma(t)^{n-k+1}}\frac{\gamma(s)}{\gamma(t)}\mathrm{d}s\\
=&\sum_{k=1}^{n}{n\choose k}
\frac{\varphi_{n-k}(t)((n-k)!)^2}{\gamma(t)^{n-k+1}}\int_0^t|\partial_x^k[(\partial_1^2\phi)(X,s)]|\gamma(s)\frac{t-s}{\gamma(t)}\mathrm{d}s\\
\leq&\sum_{k=1}^{n}{n\choose k}
\frac{\varphi_{n-k}(t)((n-k)!)^2}{\gamma(t)^{n-k+1}}\frac{(k!)^2\varphi_k(t)}{2\gamma(t)^k}\\
=&(n!)^2\sum_{k=1}^{n}{n\choose k}^{-1}
\frac{\varphi_{n-k}(t)\varphi_k(t)}{2\gamma(t)^{n+1}}\leq(n!)^2
\frac{5}{3}\cdot\frac{\varphi_{n}(t)}{2\gamma(t)^{n+1}}\leq
\frac{(n!)^2\varphi_{n}(t)}{\gamma(t)^{n+1}}.
\end{align*}This completes the proof.
\end{proof}
\begin{lemma}\label{lem1.8}Assume \eqref{rho2},
then we have\begin{align}\label{X3}
&|\partial_x^n\partial_{x_0}X(s;x,x_0,t)|\leq \frac{\varphi_n(t)(n!)^2(t-s)}{\gamma(t)^{n+1}},\quad \forall\ n\in\Z,\ n\geq0,\ 0\leq s\leq t,\\ \label{X4}
&|\partial_x^n\partial_{x_0}w(x,x_0,t)|\leq \frac{\varphi_n(t)(n!)^2}{\gamma(t)^{n+1}},\quad \forall\ n\in\Z,\ n\geq0,\ t\geq0.
\end{align}Moreover for $n\in\Z$, $n\geq0$, $t\geq0$ we have\begin{align}\label{f02}
&\int_{\R}|\partial_x^n[\tilde{f}_0(x_0,w_0(x,x_0,t))\partial_{x_0}w(x,x_0,t)]|\mathrm{d}x_0\leq\frac{(n!)^2\varphi_n(t)}{\ACD\gamma(t)^n}.
\end{align}\end{lemma}Now Proposition \ref{prop1} follows from \eqref{rho1} and \eqref{f02}.\begin{proof}
We first prove \eqref{X3} by induction on $n$. The case $n=0$ follows from Lemma \ref{lem1.4} (recall that $\varphi_0(t)=1$).
For $n\in\Z_+$, it is enough to prove \eqref{X3} assuming \eqref{X8}.
By \eqref{X5}   we have $\partial_s^2\partial_x^n\partial_{x_0}X=-q\partial_x^n[(\partial_1^2\phi)(X,s)\partial_{x_0}X]$.
For fixed $ x,x_0,t$ and $n\in\Z_+$, let $h(s):=-q(\partial_1^2\phi)(X(s;x,x_0,t),s)$, $y(s):=\partial_x^n\partial_{x_0}X(s;x,x_0,t) $, then we have
$y''(s)=h(s)y(s)+F(s)$ for $s\in[0,t]$ with $F(s):=-q(\partial_x^n[(\partial_1\phi)(X,s)\partial_{x_0}X]-(\partial_1^2\phi)(X,s)\partial_x^n\partial_{x_0}X)$.

 By Lemma \ref{lem1.2}, we have $|h(s)|\leq -\gamma''(s)/\gamma(s)$ for $s\geq0$, by \eqref{X6} we have $y(0)=y'(0)$, $y(t)=0$.
 By Lemma \ref{lem1.3} (iii) and Lemma \ref{lem1.5} (assuming \eqref{X8}) we have \begin{align*}
\sup_{s\in[0,t)}\frac{|y(s)|}{t-s}&\leq 
\int_{0}^{t}|F(s)|\frac{\gamma(s)}{\gamma(t)}\mathrm{d}s=
\int_0^t|\partial_x^n[(\partial_1\phi)(X,s)\partial_{x_0}X]-(\partial_1^2\phi)(X,s)\partial_x^n\partial_{x_0}X|\frac{\gamma(s)}{\gamma(t)}\mathrm{d}s\\
&\leq\frac{(n!)^2\varphi_n(t)}{\gamma(t)^{n+1}},\quad y(t)=0.
\end{align*}This implies \eqref{X3} (assuming \eqref{X8}). By induction, we complete the proof of \eqref{X3}.

Recall that $w(x,x_0,t)=V(t;x,x_0,t)$, $V=\partial_{s}X$, then we have $\partial_x^n\partial_{x_0}w(x,x_0,t)=y'(t) $ with $y(s):=\partial_x^n\partial_{x_0}X(s;x,x_0,t) $ and $y(t)=0$ (using \eqref{X6}). Thus  (using \eqref{X3})\begin{align*}
|\partial_x^n\partial_{x_0}w(x,x_0,t)|=|y'(t)|\leq \sup_{s\in[0,t)}\frac{|y(s)|}{t-s}
\leq\frac{(n!)^2\varphi_n(t)}{\gamma(t)^{n+1}},
\end{align*}
which is \eqref{X4}. By the Leibniz formula, \eqref{X4}, \eqref{f01} in Lemma \ref{lem1.6} and  Lemma \ref{lem1.2b} we have
\begin{align*}
&\int_{\R}|\partial_x^n[\tilde{f}_0(x_0,w_0(x,x_0,t))\partial_{x_0}w(x,x_0,t)]|\mathrm{d}x_0\\
\leq&\sum_{k=0}^{n}{n\choose k}
\int_{\R}|\partial_x^k[\tilde{f}_0(x_0,w_0(x,x_0,t))]||\partial_x^{n-k}\partial_{x_0}w(x,x_0,t)|\mathrm{d}x_0\\
\leq&\sum_{k=0}^{n}{n\choose k}
\int_{\R}|\partial_x^k[\tilde{f}_0(x_0,w_0(x,x_0,t))]|\frac{\varphi_{n-k}(t)((n-k)!)^2}{\gamma(t)^{n-k+1}}\mathrm{d}x_0\\
\leq&\sum_{k=0}^{n}{n\choose k}
\frac{(k!)^2\varphi_k(t)}{\ABD\gamma(t)^k}\frac{\varphi_{n-k}(t)((n-k)!)^2}{\gamma(t)^{n-k+1}}
=(n!)^2\sum_{k=0}^{n}{n\choose k}^{-1}
\frac{\varphi_{n-k}(t)\varphi_k(t)}{\ABD\gamma(t)^{n+1}}\\ \leq&(n!)^2
\cdot\frac{8}{3}\cdot\frac{\varphi_{n}(t)}{\ABD\gamma(t)^{n+1}}=
\frac{(n!)^2\varphi_{n}(t)}{\ACD\gamma(t)^{n+1}}, 
\end{align*}which is \eqref{f02}. This completes the proof.\end{proof}
\if0\begin{align*}
&\int_0^t \|\partial_x^{n+1}\phi(s)\|_{L^{\infty}}\gamma(s)^n\mathrm{d}s\leq \frac{\varphi_n(t)(n!)^2}{\AAC}.
\end{align*}\begin{align*}
\partial_s^2\partial_x^nX=-q\partial_x^n[(\partial_1\phi)(X,s)],\quad |\partial_x^nX(s)|\leq (1+s)\int_0^t|\partial_x^n[(\partial_1\phi)(X,s)]|\mathrm{d}s.
\end{align*}\begin{align*}
|\partial_x^n[(\partial_1\phi)(X,s)]|&\leq\sum_{*}\frac{n!}{m_1!\cdots m_n!}|(\partial_1^{k+1}\phi)(X,s)|\cdot\prod_{j=1}^n|\partial_x^jX/j!|^{m_j}\\
&\leq\sum_{*}\frac{n!}{m_1!\cdots m_n!}\|\partial_1^{k+1}\phi(s)\|_{L^{\infty}}\cdot\left|\frac{\gamma(s)}{\gamma(t)}\right|^{m_1}\cdot
\prod_{j=2}^n\left|\frac{\varphi_j(t)j!\gamma(s)}{\AAB\gamma(t)^{j}}\right|^{m_j}\\ &=\sum_{*}\frac{n!}{m_1!\cdots m_n!}\|\partial_1^{k+1}\phi(s)\|_{L^{\infty}}\frac{\gamma(s)^k}{\gamma(t)^n}\cdot
\prod_{j=2}^n\left|\frac{\varphi_j(t)j!}{\AAB}\right|^{m_j}.
\end{align*}\begin{align*}
\int_0^t|\partial_x^n[(\partial_1\phi)(X,s)]|\mathrm{d}s&\leq\sum_{*}\frac{n!}{m_1!\cdots m_n!}\frac{\varphi_k(t)(k!)^2}{\AAC\gamma(t)^n}\cdot
\prod_{j=2}^n\left|\frac{\varphi_j(t)j!}{\AAB}\right|^{m_j}\\ 
&\leq\sum_{*}\frac{(n!)^2\varphi_n(t)}{\AAC \gamma(t)^n}\cdot
\prod_{j=2}^n\frac{(j/n)^{(j-2)m_j/4}}{\ABB^{m_j}m_j!}\\ 
&\leq\sum_{m_2=0}^{+\infty}\cdots\sum_{m_n=0}^{+\infty}\frac{(n!)^2\varphi_n(t)}{\AAC\gamma(t)^n}\cdot
\prod_{j=2}^n\frac{(j/n)^{(j-2)m_j/4}}{\ABB^{m_j}m_j!}\\ 
&=\frac{(n!)^2\varphi_n(t)}{\AAC\gamma(t)^n}\cdot
\prod_{j=2}^n\exp\{(j/n)^{(j-2)/4}/\ABB\}\\ 
&\leq\frac{(n!)^2\varphi_n(t)}{\AAC\gamma(t)^n}\exp\{15/\ABB\}\leq\frac{(n!)^2\varphi_n(t)}{\ABC\gamma(t)^n}.
\end{align*}\begin{align*}
|\partial_x^nX(s)|\leq (1+s)\frac{(n!)^2\varphi_n(t)}{\ABC\gamma(t)^n}\leq \frac{20}{19}\gamma(s)\frac{(n!)^2\varphi_n(t)}{\ABC\gamma(t)^n}\leq 
\frac{(n!)^2\varphi_n(t)\gamma(s)}{\AAB\gamma(t)^n}.
\end{align*}\begin{align*}
&|\partial_xw_0(x,x_0,t)|\leq \frac{1}{\gamma(t)},\quad |\partial_x^nw_0(x,x_0,t)|\leq \frac{\varphi_n(t)(n!)^2}{\AAB\gamma(t)^{n}},\ n\geq2,\end{align*}\begin{align*}
\partial_s^2\partial_{x_0}X=-q(\partial_1^2\phi)(X,s)\partial_{x_0}X,\quad 0<\partial_{x_0}X(s;x,x_0,t)\leq (t-s)/\gamma(t).
\end{align*}\begin{align*}
\partial_s^2\partial_x^n\partial_{x_0}X=-q\partial_x^n[(\partial_1^2\phi)(X,s)\partial_{x_0}X],\quad |\partial_x^nX(s)|\leq (t-s)\int_0^t|\partial_x^n[(\partial_1^2\phi)(X,s)\partial_{x_0}X]|\frac{s+1}{t+1}\mathrm{d}s.
\end{align*}\begin{align*}
&\int_0^t \min\left\{2\varphi_{n}(s),\frac{\varphi_{n+2}(s)((n+1)(n+2))^2}{\gamma(s)^{2}}\right\}(t-s)\mathrm{d}s\leq 500t\varphi_{n}(t).
\end{align*}\begin{align*}
&\int_0^t \|\partial_x^{n+2}\phi(s)\|_{L^{\infty}}\gamma(s)^{n+1}(t-s)\mathrm{d}s\leq \frac{t\varphi_n(t)(n!)^2}{2}.
\end{align*}\begin{align*}
\int_0^t|\partial_x^n[(\partial_1^2\phi)(X,s)]|\gamma(s)(t-s)\mathrm{d}s&\leq\sum_{*}\frac{n!t\varphi_k(t)(k!)^2}{3m_1!\cdots m_n!}\frac{1}{\gamma(t)^n}\cdot
\prod_{j=2}^n\left|\frac{\varphi_j(t)j!}{\AAB}\right|^{m_j}\\
&\leq\frac{(n!)^2t\varphi_n(t)}{2.5\gamma(t)^n}.
\end{align*}\fi


\begin{thebibliography}{99}
\bibitem{BD}C. Bardos, P. Degond, {\it Global existence for the Vlasov–Poisson equation in 3 space variables with small initial data}, Ann.
 Inst. H. Poincaré Anal. Non Linéaire 2(2):101–118, 1985.
\bibitem{B}J. Bedrossian, {\it A brief introduction to the mathematics of Landau damping}, arXiv:2211.13707.
\bibitem{BMM1}J. Bedrossian, N. Masmoudi, C. Mouhot, {\it Landau damping: paraproducts and Gevrey regularity}, Ann.
 PDE, 2, pp. Art. 4, 71, 2016
\bibitem{BMM2}J. Bedrossian, N. Masmoudi, C. Mouhot, {\it Landau damping in finite regularity for unconfined systems with screened interactions}, 
    Comm. Pure Appl. Math., 71, 537–576, 2018.
\bibitem{BM} J. Bedrossian, N. Masmoudi, {\it Inviscid damping and the asymptotic stability of planar shear flows in the 2D Euler equations},
Publ. Math. Inst. Hautes \'{E}tudes Sci., 122, 195-300, 2015.
\bibitem{CHL}S.-H. Choi, S.-Y. Ha, H. Lee, {\it Dispersion estimates for the two-dimensional Vlasov-Yukawa
 system with small data}, J. Differential Equations 250(1):515–550, 2011.
\bibitem{D}X. Duan, {\it Sharp Decay Estimates for the Vlasov-Poisson and Vlasov-Yukawa Systems with Small Data}, Kinet. Relat. Models 
15(1):119–146, 2022.
\bibitem{GS1}    R. Glassey, W. Strauss. {\it Singularity formation in a collisionless plasma could occur only
 at high velocities}. Arch. Rational Mech. Anal., 92(1):59–90, 1986.
\bibitem{GS2}    R. Glassey, W. Strauss. {\it High velocity particles in a collisionless plasma}. Math. Methods
 Appl. Sci., 9(1):46–52, 1987.
\bibitem{GS3}   R. Glassey, W. Strauss. {\it Large velocities in the relativistic Vlasov-Maxwell equations}. J.
 Fac. Sci. Univ. Tokyo Sect. IA Math., 36(3):615–627, 1989.
\bibitem{HRV}H. Hwang, A. Rendall, J. Vel\'azquez, {\it Optimal gradient estimates and asymptotic behaviour for
 the Vlasov-Poisson system with small initial data}, Arch. Ration. Mech. Anal. 200(1):313–360, 2011.
\bibitem{IRW}M. Iacobelli, S. Rossi, K. Widmayer, {\it On the stability of vacuum in the screened Vlasov-Poisson equation}, arXiv:2410.17978.
\bibitem{IJ1} A. Ionescu, H. Jia, {\it Inviscid damping near the Couette flow in a channel}, Comm. Math. Phys., 374:2015-2096, 2020.
\bibitem{IJ2} A. Ionescu, H. Jia, {\it Nonlinear inviscid damping near monotonic shear flows}, Acta Math. 230:321–399, 2023.
\bibitem{IPWW}A. D. Ionescu, B. Pausader, X. Wang, K. Widmayer, {\it On the Asymptotic Behavior of Solutions to the
 Vlasov–Poisson System}, International Mathematics Research Notices. IMRN, 8865–8889, 2022.
\bibitem{IPWW2}A. D. Ionescu, B. Pausader, X. Wang, K. Widmayer, {\it Nonlinear Landau damping and wave operators in sharp Gevrey spaces},  arXiv:2405.04473.
\bibitem{MZ} N. Masmoudi, Z. Zhao, {\it Nonlinear inviscid damping for a class of monotone shear flows in finite
 channel}, Ann. of Math. (2) 199:1093–1175, 2024.
\bibitem{MV} C. Mouhot, C. Villani, {\it On Landau damping}, Acta Math., 207:29–201, 2011.
\bibitem{WY}D. Wei, S. Yang, {\it On the 3D relativistic Vlasov-Maxwell system with large Maxwell field}, Comm. Math. Phys. 383(3):2275–2307, 2021.
\bibitem{Y} H. Yukawa, {\it On the interaction of elementary particles}, Proc. Phys. Math. Soc. Japan 17:48–57, 1935.
\end{thebibliography}
\end{document}